\documentclass[11pt]{article}

\usepackage{amsmath, amssymb, amsthm}
\usepackage[utf8]{inputenc}
\usepackage[T1]{fontenc}
\usepackage{geometry}
\usepackage{hyperref}
\usepackage{mathtools}

\newtheorem{theorem}{Theorem}[section]
\newtheorem{lemma}[theorem]{Lemma}
\newtheorem{proposition}[theorem]{Proposition}
\newtheorem{corollary}[theorem]{Corollary}
\theoremstyle{definition}
\newtheorem{definition}[theorem]{Definition}
\newtheorem{example}[theorem]{Example}
\theoremstyle{remark}
\newtheorem{remark}[theorem]{Remark}

\newcommand{\C}{\mathbb{C}}
\newcommand{\R}{\mathbb{R}}
\newcommand{\Z}{\mathbb{Z}}
\newcommand{\D}{\mathbb{D}}
\newcommand{\Acal}{\mathcal{A}}
\newcommand{\Bcal}{\mathcal{B}}
\newcommand{\Ecal}{\mathcal{E}}
\newcommand{\Mcal}{\mathcal{M}}

\newcommand{\fa}{\mathfrak{a}}
\newcommand{\fb}{\mathfrak{b}}
\newcommand{\ff}{\mathfrak{f}}
\newcommand{\Mob}{\mathrm{M\ddot{o}b}(\D)}
\newcommand{\hyparea}{\omega_{\mathrm{hyp}}}
\newcommand{\Sol}{\operatorname{Sol}}

\title{Moduli of the Sourceless Framed Beltrami--Vekua Normal Form}
\author{Daniel Alay\'on-Solarz\thanks{\texttt{danieldaniel@gmail.com}}}
\date{August 2026}

\begin{document}

\maketitle

\begin{flushright}
\begin{minipage}{0.6\textwidth}
\begin{flushright}
\small\itshape
We found that objectivity means invariance with respect to the
group of automorphisms.
\par\vspace{2pt}
{\normalfont\small --- Hermann Weyl, \emph{Symmetry} (1952), p.~132}
\end{flushright}
\end{minipage}
\end{flushright}
\vspace{6pt}

\begin{abstract}
We study the moduli of sourceless framed Beltrami--Vekua equations under recombinations of the unknown, scalings, and orientation-preserving changes of variables. On every bounded simply connected domain, such an equation reduces to $w_{\bar z}=B\bar w$ on the unit disk, with residual symmetries given exactly by zero-free holomorphic gauges and M\"obius transformations. When $B$ is zero-free, the equation is completely classified by two data modulo M\"obius: the hyperbolic mass density $\vartheta=\tfrac14(1-|z|^2)^2|B|^2$ and the phase-curvature current $K=\Delta\arg B$, whose hyperbolic density at $C^2$ regularity is $\kappa=\Delta_{\mathrm{hyp}}\arg B$. We determine the exact range of these invariants: every positive H\"older density and every phase current arising as the Laplacian of a H\"older phase occur. For fields with zeros, the classification extends to the phase-integrable sector through the triple $(\vartheta,d\eta,d{\star}\eta)$, where $\eta=\operatorname{Im}(dB/B)$. On the tame sector, $d\eta$ is the atomic charge measure, while $d{\star}\eta$ carries the remaining phase curvature. Thus the pseudo-analytic mass and charge are numerical projections of a larger infinite-dimensional moduli space. Explicit equal-mass, equal-charge, inequivalent equations are exhibited.
\end{abstract}

\section{Introduction}\label{sec:intro}

A first-order real planar elliptic system normalizes, by pointwise
algebra alone, to a \emph{framed Beltrami--Vekua equation}
\begin{equation}\label{eq:framed}
\Phi\,(w_{\bar z} - \mu\, w_z) \;+\; \Psi\,(\overline{w_z} - \mu\,\overline{w_{\bar z}}) \;+\; \fa\, w \;+\; \fb\, \bar w \;=\; \ff,
\qquad |\mu| < 1, \quad |\Phi| > |\Psi|,
\end{equation}
the terminal class closed under the recombinations of the unknown
$w = \varphi w' + \psi\bar w'$, the scalings of the equation, and
orientation-preserving changes of variables \cite{framed, absorption}. On the
subclass whose frame ratio carries a derivative --- the standing
hypothesis of the program, weighed below --- two invariants are known
on their respective domains. The \emph{pseudo-analytic mass}
$\Mcal \in [0, \infty]$ integrates the squared modulus of the
numerator field $N$ against the invariant weight built from the frame
determinant and the conformal factor of $\mu$ \cite{mass, framed}. On
a $C^1$ frame representative,
$N = \Phi\fb - \Psi\fa - W_L(\Phi, \Psi)$, $L := \bar \partial - \mu \partial$; for a projectively $C^1$
frame it is read by the de-scaling convention fixed below. When the
zero set of $N$ is charge-admissible --- in particular, when it is
compactly contained in the domain --- the \emph{pseudo-analytic
charge} $n \in \Z$ winds its argument along an enclosing curve
\cite{charge}. On this common domain mass and charge are independent
--- every pair in $(0, \infty) \times \Z$ is realized \cite{charge}
--- and the question this paper answers is the natural next one:
\emph{is the pair $(\Mcal, n)$ complete?} That is, restricting to the
sourceless case $\ff = 0$ --- a condition preserved by all three
actions, so that the sourceless equations form an invariant subclass
--- do equal mass and equal charge force equivalence?

The answer is no, and by a wide margin; but the computation that
produces the answer produces more, namely, on simply connected
domains, the full moduli of the vortex-free sourceless class ---
completeness and realization both --- together with complete orbit
invariants across the vortices, on the phase-integrable
sector. The moduli turn out to
be a \emph{pair of scalar data} on the unit disk modulo the
three-dimensional M\"obius group --- an infinite-dimensional space ---
on the vortex-free sector, the second datum a current at the bare
regularity of the class, enlarged across the vortices by a third,
the charge current, purely atomic on the tame sector. The mass is the
zeroth moment of the first field; whenever the numerical charge is
defined, it is the quantized total flux of the charge current, a
functional of the third datum alone and identically zero on the
vortex-free sector. The pair
$(\Mcal, n)$ is not complete for the same reason a function is not
determined by its integral.

\medskip\noindent\textbf{The minimal reduction.}
The reduction chain is assembled from the companion papers, run here
in the order that spends least regularity \cite[Rem.~10.11]{framed}. A
sourceless framed equation on a bounded simply connected domain is
carried, by one substitution and one scaling, onto the trivial-frame
slice --- the Beltrami--Vekua equation of \cite{mass} --- with the
$L$-Wronskian of the frame absorbed into the coefficient of $\bar w$
(the straightening of \cite[\S 9]{framed}); by the uniformization of
its Beltrami coefficient and a Riemann map, onto the slice $\mu = 0$
over the unit disk; and by one multiplicative gauge, solving
$\bar\partial s = -\Acal$ by the Pompeiu integral, onto the
\emph{minimal form}
\begin{equation}\label{eq:minimal}
w_{\bar z} \;=\; B\,\bar w \qquad \text{on } \D .
\end{equation}
Every datum of the original equation has been consumed except one
complex-valued function: the coefficient $B$, which is the numerator
field $N$ of the representative \eqref{eq:minimal}. The chain is the
composite already implicit in \cite{framed, charge}; what is new is
the bookkeeping of what survives it.

\medskip\noindent\textbf{The residual symmetries, and a phase that is
pure gauge.} Two equations are equivalent if and only if their
minimal forms are related by the symmetries that preserve the
minimal slice, and these form a transparent group: the gauges
preserving $\Acal = 0$ are $w = \varphi w'$ with $\varphi$
\emph{holomorphic} and zero-free, acting by
$B \mapsto (\bar\varphi/\varphi)\,B$; and the changes of variables
preserving $\mu = 0$ are conformal automorphisms $F \in \Mob$, acting
by $B \mapsto \overline{F'}\,(B \circ F)$
(Proposition~\ref{prop:residual}). One small observation now does
disproportionate work: $F'$ is holomorphic and zero-free on the simply
connected $\D$, so it has a holomorphic square root, and the phase of
the conformal weight $\overline{F'}$ is itself a residual gauge. The
orbit relation therefore splits into a modulus law and a phase law,
\begin{equation}\label{eq:orbit-split}
|B'| \;=\; |F'|\;|B \circ F|, \qquad
\arg B' \;=\; \arg B \circ F \;+\; h, \quad h \ \text{an arbitrary harmonic function},
\end{equation}
the harmonic functions being exactly the phases $-2\arg\varphi$ of the
zero-free holomorphic gauges.

\medskip\noindent\textbf{The moduli pair.} Each half of
\eqref{eq:orbit-split} is neutralized by the hyperbolic geometry of
$\D$, through the identity $1 - |F(z)|^2 = |F'(z)|\,(1 - |z|^2)$ valid
for every $F \in \Mob$. The modulus law is absorbed by the conformal
factor: the \emph{hyperbolic mass density}
\[
\vartheta \;:=\; \tfrac14\,\bigl(1 - |z|^2\bigr)^2\,|B|^2
\]
is gauge-invariant and satisfies the exact equivariance
$\vartheta' = \vartheta \circ F$. The phase law is absorbed by the
Laplacian: harmonic additions die under $\Delta$, and conformal
covariance gives the \emph{phase-curvature current}
\[
K \;:=\; \Delta\arg B\;dx\,dy,
\]
well-defined on the vortex-free sector at the bare regularity of the
class, gauge-invariant, and transformed by pullback, $K' = F^*K$.
When the phase admits a $C^2$ branch, this current has hyperbolic
density
\[
\kappa \;:=\; \tfrac14\,\bigl(1 - |z|^2\bigr)^2\,\Delta\arg B,
\qquad K = \kappa\,\hyparea,
\]
and $\kappa' = \kappa \circ F$ (Theorem~\ref{thm:equivariance} and
Remark~\ref{rem:distributional}). The modulus field and phase current
are the same kind of covariant object in different clothing:
$\vartheta$ is the mass 2-form $\Theta$ of \cite{mass, framed} read
against hyperbolic area, while $K$ records the phase that the mass
forgets.

\medskip\noindent\textbf{The main theorem.} On the sector where $B$ is
zero-free --- equivalently, where the equation is everywhere
irreducibly pseudo-analytic in the sense of \cite[Rem.~5.4]{charge}
--- the pair is complete: two sourceless equations are equivalent if
and only if their pairs $(\vartheta, K)$ agree up to a common
M\"obius transformation of $\D$ --- the phase read through the
current $K$, whose hyperbolic density on the $C^2$ sector is $\kappa$
(Theorem~\ref{thm:completeness}). The proof is one page once the
reduction is in place: matching $\vartheta$ matches $|B|$; matching
$K$ makes the phase difference distributionally harmonic; and a
harmonic phase difference is a gauge. The moduli of the sourceless vortex-free class
are therefore the pairs $(\vartheta, K)$ modulo the three-parameter group
$\Mob$ --- an infinite-dimensional space, of which the invariants
$(\Mcal, n)$ retain two real numbers.

\medskip\noindent\textbf{Mass and charge as projections of the moduli.}
The relation of the moduli data to the mass and to the numerical charge
on its charge-admissible domain is exact. The mass is the zeroth
hyperbolic moment,
\[
\Mcal \;=\; \int_\D \vartheta\; \hyparea, \qquad \hyparea := \frac{4\, dx\, dy}{(1 - |z|^2)^2},
\]
and more generally every functional $\int_\D \Phi(\vartheta)\,\hyparea$
is an invariant, of which the mass is the case $\Phi = \mathrm{id}$
(Section~\ref{sec:massandcharge}). The charge lives in the vortex
sector, where $\arg B$ ceases to be single-valued and survives as the
\emph{phase form} $\eta = d\arg B$. The derivative of the phase then
splits into two currents: the exterior derivative $d\eta$ is the
charge current --- on the tame sector it is purely atomic,
$2\pi n_p$ at each vortex --- while the co-derivative
$d{\star}\eta$ extends the phase curvature across the vortices. On the tame sector the latter has no
atomic part; on the general phase-integrable sector both objects are
currents and may have singular components (Section~\ref{sec:vortices}). Incompleteness is then
exhibited twice
over, at equal mass and equal charge
(Section~\ref{sec:incompleteness}): by two vortex-free equations with
the same $(\Mcal, n)$ and identically vanishing phase curvature,
separated by the second moment of $\vartheta$ --- the fields $B = t$
and $B = \sqrt{2}\,t\,(1 - |z|^2)^{1/2}$ have equal mass $\pi t^2$ and
charge $0$, but second moments $\pi t^4/12$ and $\pi t^4/5$ --- and
by two equations with identical $\vartheta$, mass $\pi$, charge $0$,
separated by the phase curvature: $B = 1$ has $\kappa \equiv 0$, while
$B = e^{ix^2}$ has $\kappa = \tfrac12(1 - |z|^2)^2 > 0$, and the exact
law $\kappa' = \kappa \circ F$ can never carry a nonvanishing field to
zero.

\medskip\noindent\textbf{Across the vortices.} The vortex-free
restriction of the main theorem is not final. On the sector where the
phase form remains locally integrable --- in particular at every tame
vortex $B = (z - p)^{n_p} g$ with $g$ zero-free --- completeness
persists, with the pair enlarged to the triple; phase-integrability,
as Definition~\ref{def:phase-form} fixes it, carries the standing
hypotheses that the zero set be closed, Lebesgue-null, and of
connected complement, the last used essentially in the sufficiency
proof, and its relaxation is the open problem~(3) of
Section~\ref{sec:discussion}. The enlarged data are
$(\vartheta, d\eta, d{\star}\eta)$
(Theorem~\ref{thm:tame-complete}); and the charge current is
irreducibly needed there, the fields $z$ and $\bar z$ sharing
$\vartheta$ and curvature while carrying opposite charge. The frontier of the order-zero $L^1$ current classification is the
integrability of the phase form: the oscillation
example of \cite[Rem.~4.8]{charge} --- a vortex of local charge zero
whose phase oscillates without limit --- has
$|\eta| \sim |z - p|^{-2}$, hence no locally integrable phase form
and no order-zero currents, and beyond that
frontier the classification of the vortex germs remains open
(Section~\ref{sec:vortices}). Terminology is fixed accordingly and
held rigid throughout: the \emph{vortex-free moduli} are the pair
$(\vartheta, K)$; the \emph{phase-integrable moduli} are the triple
$(\vartheta,\, d\eta,\, d{\star}\eta)$, the charge carried by $d\eta$
alone.

\medskip\noindent\textbf{The class, and Bojarski's shadow.} A scope
distinction runs through the program and should be visible here. The
framed normalization itself is universal: pointwise algebra, taking no
derivative of any datum, available at bare continuity. The invariants
are not so free. The numerator field differentiates a de-scaled frame through
the $L$-Wronskian, and with it everything in this paper --- mass,
charge, moduli --- lives on the strictly smaller class whose
projective frame ratio $\nu = \Psi/\Phi$ carries one derivative, the
common factor being scaling data. And the demand grades by invariant: the mass
survives that derivative being merely weak --- frame in
$W^{1,2}_{\mathrm{loc}} \cap L^{\infty}_{\mathrm{loc}}$, the
measurable frontier of \cite{framed} --- while the charge is
established only for a classical derivative, the frame ratio $C^1$
\cite{charge}, its measurable extension being precisely the open
conjecture of \cite[\S 7]{charge}; the moduli of this paper run in the
H\"older class of Definition~\ref{def:class}, and
Section~\ref{sec:discussion} records what descends. The comparison
is with Bojarski \cite{bojarski}: his reduction of the rough
system to $w_{\bar z} = q_1 w_z + q_2\,\overline{w_z}$ differentiates
no coefficient and runs at bare measurable principal part, while the
frame is principal data and the framed theory demands a derivative of
it. The demand is the exact price of the zeroth-order invariants,
which a derivative-free reduction cannot see; and the first
expenditure of the ledger below is this same derivative, already
present in $N$, not a new cost of the chain.

\medskip\noindent\textbf{The regularity ledger.} The chain spends: one
derivative on the projective frame ratio (the straightening, as in
\cite{framed}); one
Beltrami solve (the uniformization --- locally H\"older $\mu$ buys the
$C^{1,\alpha}_{\mathrm{loc}}$ chart that keeps the category smooth);
and one $\bar\partial$-solve (the gauge removing $\Acal$ --- locally
the Pompeiu integral, globally an elliptic solve). We therefore work
in the locally H\"older category throughout ---
Definition~\ref{def:class} fixes the classes, local in everything but
the ellipticity bound on $\mu$ --- and record at the end what survives
below it, where the phase half of the story meets the open trace
problems of \cite[\S 7]{charge}.

\medskip\noindent\textbf{Outline.} Section~\ref{sec:reduction} runs
the minimal reduction and computes the residual groupoid.
Section~\ref{sec:pair} defines the current-valued pair
$(\vartheta, K)$, with $\kappa$ as the classical hyperbolic density of
$K$ on the $C^2$ sector, and proves the equivariance. Section~\ref{sec:completeness} proves the
completeness theorem on the vortex-free sector and discusses
realization. Section~\ref{sec:massandcharge} identifies mass and
charge as projections of the moduli data.
Section~\ref{sec:incompleteness} exhibits the incompleteness of
$(\Mcal, n)$. Section~\ref{sec:vortices} splits the derivative of
the phase into its two currents and extends the completeness theorem
across the vortices. Section~\ref{sec:solutions} reads the
classification at the level of solutions: distinct minimal forms
share no nonzero solution, the sheaf of local solutions determines
the equation unconditionally, the pair is computable from any one
nonzero solution, and equivalence is the existence of a
multiplier--composition isomorphism of the solution sheaves.
Section~\ref{sec:discussion} closes with the
physical reading of the moduli --- energy profile, flat
Aharonov--Bohm connection, co-curvature --- the locality ledger of
the invariants, and the open problems
on multiply connected domains and at measurable regularity; a coda,
Section~\ref{sec:coda}, states what the program has computed.

\section{The minimal reduction}\label{sec:reduction}

This section assembles, from \cite{mass, framed}, the reduction of a
sourceless framed equation to the minimal form
\eqref{eq:minimal}, and computes the residual symmetries exactly. The
name is the count: the reduction consumes every datum of the original
equation except the single function $B$, and nothing smaller survives
it. The
reduction itself is the composite of results proved in the companion
papers; the content here is the precise statement of what preserves
the minimal slice, on which everything that follows rests.

\begin{definition}[The class]\label{def:class}
Fix $\alpha \in (0, 1)$. Throughout Sections
\ref{sec:reduction}--\ref{sec:incompleteness}, a \emph{sourceless
framed equation} is an equation \eqref{eq:framed} with $\ff = 0$ on a
bounded simply connected domain $\Omega \subset \C$, with
\[
\Phi,\Psi \in C^{\alpha}_{\mathrm{loc}}(\Omega), \qquad
|\Phi|>|\Psi|, \qquad
\nu:=\frac{\Psi}{\Phi}\in C^{1,\alpha}_{\mathrm{loc}}(\Omega),
\]
$\mu \in C^{\alpha}_{\mathrm{loc}}(\Omega)$ with
$\sup_\Omega |\mu| < 1$, and
$\fa, \fb \in C^{\alpha}_{\mathrm{loc}}(\Omega)$. Thus the derivative
hypothesis is projective: it is the frame ratio $\nu$, rather than an
arbitrary common scaling factor of $(\Phi,\Psi)$, that carries one
derivative. Whenever a formula differentiates the frame, it is read
on the canonical de-scaled representative $\Phi^{-1}\Ecal$, whose
frame is $(1,\nu)$, or on any equivalent $C^{1,\alpha}$ de-scaling.
Write $N_0$ for the numerator of the canonical representative and
define the numerator of the displayed representative by
$N:=\Phi^2N_0$. On a displayed $C^{1,\alpha}$ frame this agrees with
$N=\Phi\fb-\Psi\fa-W_L(\Phi,\Psi)$, and under a scaling by $c$ it
obeys $N\mapsto c^2N$. Thus the definition is independent of the
de-scaling and has exactly the scaling law used in
\cite{framed, charge}.

The admissible morphisms are those of \cite{framed, charge}:
substitutions $w = \varphi w' + \psi\bar w'$ with
$\varphi, \psi \in C^{1,\alpha}_{\mathrm{loc}}$,
$|\varphi| > |\psi|$ pointwise; scalings $\Ecal \mapsto c\,\Ecal$ with
$c \in C^{\alpha}_{\mathrm{loc}}(\Omega; \C^*)$; and
orientation-preserving $C^{1,\alpha}_{\mathrm{loc}}$ diffeomorphisms.
Two equations are \emph{equivalent} if a finite composite of
admissible morphisms carries one to the other. The projective
regularity is preserved by all three actions: scalings leave $\nu$
unchanged, substitutions act on it through $C^{1,\alpha}$ data, and
changes of variables compose it with a $C^{1,\alpha}$ diffeomorphism.
The regularity is local and the ellipticity bound on $\mu$ alone is
global; as Theorem~\ref{thm:reduction} and
Proposition~\ref{prop:realization} will show, the class is closed under
its own minimal reduction --- the minimal forms it produces are again
members --- the boundary of $\Omega$ playing no role in any
construction of this paper. The calligraphic letter is deliberate:
the roman $E$ of \cite{framed} denotes the complex residual of the
underlying real system --- the framed equation arising as the
recombination $E + \mu\bar E$ --- and is not reused here. All three
actions are orientation-preserving; the equivalence of this paper is
the orientation-preserving groupoid throughout, and the effect of
adjoining orientation-reversing morphisms --- at the minimal slice,
a single reflection --- is not pursued here.
\end{definition}

The H\"older exponent is spent in exactly two places --- the
uniformizing chart and the $\Acal$-gauge below --- and nowhere else;
the projective frame algebra of \cite{framed} runs at
$C^{1,\alpha}$/$C^{\alpha}$ as above.

\begin{theorem}[Minimal reduction]\label{thm:reduction}
Every sourceless framed equation per Definition~\ref{def:class} is
equivalent to a minimal form
\[
w_{\bar z} \;=\; B\,\bar w \qquad \text{on } \D,
\qquad B \in C^{\alpha}_{\mathrm{loc}}(\D),
\]
whose coefficient is, up to a sign, the numerator field of the
representative: the equation $w_{\bar z} + \fa w + \fb \bar w = 0$
with trivial frame, $\mu = 0$, $\fa = 0$, $\fb = -B$ has
$N = \fb = -B$, and Remark~\ref{rem:sign} discharges the sign. We write the minimal form as
\eqref{eq:minimal} and its field as $B$, so that
$Z(B)$ is the vortex set and $|B|^2\,dx\,dy$ the mass density of the
representative.
\end{theorem}

\begin{proof}
Three steps, each from the companions.

\emph{Step 1 (straightening).} By \cite[Prop.~9.1]{framed}, the
substitution $w = w' - \nu\bar w'$ with
$\nu = \Psi/\Phi \in C^{1,\alpha}_{\mathrm{loc}}$,
followed by the scaling $c = (\Phi(1 - |\nu|^2))^{-1}$, carries the
equation onto the trivial-frame slice over the same $\mu$: the
Beltrami--Vekua equation
$w'_{\bar z} - \mu w'_z + \Acal_{\mathrm{str}} w' + \Bcal_{\mathrm{str}} \bar w' = 0$,
with $\Bcal_{\mathrm{str}} = N/(\Phi^2(1 - |\nu|^2))$ --- $N$
understood in the projective reading of Definition~\ref{def:class} ---
and $\Acal_{\mathrm{str}}, \Bcal_{\mathrm{str}} \in C^{\alpha}_{\mathrm{loc}}$,
the sourceless condition persisting since $\ff = 0$ scales to $0$.

\emph{Step 2 (uniformization).} Extend $\mu$ by $0$ to $\C$; since
$\sup_\Omega |\mu| < 1$, the principal solution of the Beltrami
equation is a quasiconformal homeomorphism of $\C$, and on $\Omega$,
where $\mu \in C^{\alpha}_{\mathrm{loc}}$, it is a
$C^{1,\alpha}_{\mathrm{loc}}$ diffeomorphism with positive Jacobian
\cite{vekua, aim}; composing with a Riemann map of the image ---
conformal, hence preserving the slice $\mu = 0$ it lands on --- and
applying the diffeomorphism law of \cite[Prop.~7.1]{framed} carries
the equation to $w_{\bar z} + \Acal w + \Bcal \bar w = 0$ on $\D$,
with $\Acal, \Bcal \in C^{\alpha}_{\mathrm{loc}}(\D)$. (The change of
variables multiplies the frame by the weight $\rho$ and is undone by
the corresponding scaling, per the reading convention of
\cite[\S 2]{charge}; on the slice $\mu = 0$ reached here, minimal
forms pull back to minimal forms directly, as Step 3 of
Proposition~\ref{prop:residual} makes explicit.)

\emph{Step 3 (the $\Acal$-gauge).} Solve $\bar\partial s = -\Acal$ on
$\D$. Locally this is the Pompeiu integral; globally,
$\Acal \in C^{\alpha}_{\mathrm{loc}}(\D)$ need not be integrable up to
the boundary, and we invoke instead the surjectivity of the elliptic
constant-coefficient operator $\bar\partial$ on $\mathcal{D}'(\D)$
--- every open planar set is $P$-convex for an elliptic $P$
\cite[Ch.~10]{hormander} --- together with elliptic regularity,
obtaining $s \in C^{1,\alpha}_{\mathrm{loc}}(\D)$. The gauge
$w = \varphi w'$, $\varphi = e^{s}$, zero-free, acts by
\cite[Prop.~5.1]{charge}:
$\Acal \mapsto \Acal + \varphi_{\bar z}/\varphi = \Acal + \bar\partial s = 0$
and $\Bcal \mapsto \Bcal\,\bar\varphi/\varphi =: -B$. The result is
\eqref{eq:minimal} with $B \in C^{\alpha}_{\mathrm{loc}}(\D)$.
\end{proof}

\begin{remark}[Sign convention]\label{rem:sign}
We write the minimal form as $w_{\bar z} = B\bar w$, so that
$B = -\Bcal$ of the slice presentation
$w_{\bar z} + \Acal w + \Bcal\bar w = 0$; since every law below is either
linear or quadratic in the field and the invariants consume $|B|$ and
$\Delta\arg B$, the sign is immaterial, and we will freely call $B$
the numerator field of the minimal representative.
\end{remark}

The equivalence of Definition~\ref{def:class} is generated by
morphisms of three types, and a finite composite of generators may in
principle pass through intermediate equations lying off any given
slice. The following lemma removes this obstacle once and for all: it
collapses every composite to one morphism of each type, so that
membership of the endpoints in a slice can be tested generator by
generator.

\begin{lemma}[Normal form of composites]\label{lem:composite}
Every finite composite of admissible morphisms equals the result of
applying, in order, a single substitution, a single scaling, and a
single orientation-preserving diffeomorphism, each admissible per
Definition~\ref{def:class}. Moreover substitutions and scalings do not
alter the Beltrami coefficient $\mu$ \cite{framed}; only the
diffeomorphism does.
\end{lemma}

\begin{proof}
Three observations. \emph{(1) Each type is closed under composition.}
Diffeomorphisms compose to diffeomorphisms and scalings compose by
multiplication of their zero-free $C^{\alpha}_{\mathrm{loc}}$ factors.
For substitutions, identify $w = \varphi w' + \psi\bar w'$ with the
field of $\R$-linear maps $v \mapsto \varphi v + \psi\bar v$, whose
determinant is $|\varphi|^2 - |\psi|^2$; the composite of two
substitutions is the substitution whose data
$(\varphi_1\varphi_2 + \psi_1\bar\psi_2,\;
\varphi_1\psi_2 + \psi_1\bar\varphi_2)$ is the pointwise composite of
the linear maps, again $C^{1,\alpha}_{\mathrm{loc}}$, and its
determinant is the product
$(|\varphi_1|^2 - |\psi_1|^2)(|\varphi_2|^2 - |\psi_2|^2) > 0$, which
for a map $v \mapsto av + b\bar v$ is exactly the frame inequality
$|a| > |b|$. \emph{(2) Scalings commute with substitutions}: the
scaling multiplies the equation, the substitution transforms the
unknown, and applying them in either order yields the same equation.
\emph{(3) Diffeomorphisms normalize the pointwise generators}: by the
transport of data, a change of variables $D$ followed by the
substitution of data $(\varphi, \psi)$ equals the substitution of data
$(\varphi \circ D, \psi \circ D)$ followed by $D$, and likewise for a
scaling $c$ with $c \circ D$; the transported data are again
admissible, $D$ being a $C^{1,\alpha}_{\mathrm{loc}}$ diffeomorphism,
and the frame inequality and zero-freeness are pointwise conditions,
preserved by composition with $D$. Hence in any finite word of
generators every diffeomorphism may be moved to the last position at
the cost of transporting the pointwise data it crosses; the
diffeomorphism block then collapses to one diffeomorphism by (1), and
the remaining pointwise block collapses, by (2) and (1), to one
substitution followed by one scaling. That substitutions and scalings
leave $\mu$ untouched is the structure of the class \eqref{eq:framed}
itself: the recombinations of the unknown and the scalings of the
equation act on the frame and the lower-order coefficients only, the
Beltrami coefficient being moved by changes of variables alone
\cite{framed}.
\end{proof}

\begin{proposition}[The residual groupoid]\label{prop:residual}
An admissible morphism carries a minimal form \eqref{eq:minimal} on
$\D$ to a minimal form on $\D$ if and only if it is a composite of:
\begin{itemize}
\item[(i)] \emph{holomorphic gauges}: $w = \varphi w'$ with $\varphi$
holomorphic and zero-free on $\D$, welded to the scaling
$c = \varphi^{-1}$, acting by
\[
B \;\longmapsto\; \frac{\bar\varphi}{\varphi}\; B ;
\]
\item[(ii)] \emph{M\"obius changes of variables}: $F \in \Mob$, acting
by
\[
B \;\longmapsto\; \overline{F'}\;\bigl(B \circ F\bigr).
\]
\end{itemize}
Consequently two minimal forms $B_1, B_2$ on $\D$ are equivalent if
and only if
\begin{equation}\label{eq:orbit}
B_1 \;=\; \frac{\bar\varphi}{\varphi}\;\overline{F'}\;\bigl(B_2 \circ F\bigr)
\qquad \text{for some } \varphi \text{ holomorphic zero-free},\ F \in \Mob .
\end{equation}
\end{proposition}

\begin{proof}
\emph{Stabilizer of the slice data.} A substitution--scaling pair
preserves the trivial frame iff it is a gauge $w = \varphi w'$,
$c = \varphi^{-1}$, with $\varphi \in C^{1,\alpha}$ zero-free
\cite[Prop.~5.1]{charge}, acting on the slice data by
$\Acal \mapsto \Acal + L\varphi/\varphi$,
$\Bcal \mapsto \Bcal\bar\varphi/\varphi$. At $\mu = 0$, preserving
$\Acal = 0$ forces $\varphi_{\bar z} = 0$: the gauge is holomorphic.
This gives (i).

\emph{Changes of variables.} Let $F : \D \to \D$ be an admissible
diffeomorphism carrying $\mu = 0$ to $\tilde\mu = 0$. The pulled-back
Beltrami coefficient is
$\tilde\mu = (F_{\bar z} + 0)/(F_z + 0) = F_{\bar z}/F_z$
\cite[Prop.~7.1]{framed}, so $\tilde\mu = 0$ forces $F_{\bar z} = 0$:
$F$ is a conformal automorphism of $\D$, i.e.\ $F \in \Mob$. Its
action on the minimal form is direct: for $h = w \circ F$, the chain
rule at $F_{\bar z} = 0$ gives
$h_{\bar z} = \overline{F_z}\,(w_{\bar\zeta} \circ F)
= \overline{F'}\,(B \circ F)\,(\bar w \circ F)
= \overline{F'}\,(B \circ F)\,\bar h$: the pullback of a minimal form
is already minimal, no scaling required, with the field of (ii).
(Equivalently: the frame weight is $\rho = F'/|F'|^2 = 1/\overline{F'}$,
and the reading convention $\tilde N = (N \circ F)/\rho$ of
\cite[\S 2]{charge} gives the same law.) This gives (ii).

\emph{Sufficiency of the composites.} Conversely, let a finite
composite of admissible morphisms carry the minimal form $B_2$ to
the minimal form $B_1$. By Lemma~\ref{lem:composite} the composite
is one substitution, then one scaling, then one diffeomorphism $D$,
and only $D$ touches the Beltrami coefficient. The intermediate
equation produced by the substitution--scaling pair therefore still
has $\mu = 0$, and $D$ carries $\mu = 0$ to $\mu = 0$: by the second
computation above, $D = F^{-1}$ for some $F \in \Mob$, and the
pullback of the minimal form $B_1$ along $F$ is again minimal.
The substitution--scaling pair thus maps the minimal form $B_2$ to a
minimal form, hence preserves the trivial frame and $\Acal = 0$, and
by the first computation it is a holomorphic gauge. Every equivalence
of minimal forms is therefore a composite of (i) and (ii), and
\eqref{eq:orbit} follows.
\end{proof}

\begin{proposition}[The phase of the weight is a gauge]\label{prop:weight-gauge}
For every $F \in \Mob$ there is a holomorphic zero-free $h$ on $\D$
with $h^2 = F'$, and
\[
\overline{F'} \;=\; \frac{\bar h}{h}\;|F'| .
\]
Consequently the orbit relation \eqref{eq:orbit} splits: $B_1$ and
$B_2$ are equivalent if and only if for some $F \in \Mob$
\begin{equation}\label{eq:orbit-split-2}
|B_1| \;=\; |F'|\;|B_2 \circ F|
\qquad\text{and}\qquad
\arg B_1 \;\equiv\; \arg(B_2 \circ F) \pmod{\text{harmonic functions}},
\end{equation}
the harmonic functions on $\D$ being exactly the phases
$-2\arg\varphi$ of the holomorphic zero-free gauges.
\end{proposition}

\begin{proof}
$F'$ is holomorphic and zero-free on the simply connected $\D$, hence
$F' = e^{g}$ with $g$ holomorphic, and $h := e^{g/2}$ is a holomorphic
square root; then
$\overline{F'} = \bar h\,\bar h = (\bar h/h)\,h\bar h = (\bar h/h)|F'|$.
The factor $\bar h/h$ is the action of the holomorphic gauge
$\varphi = h$, so absorbing it into (i) leaves the positive factor
$|F'|$ on the modulus and nothing on the phase. For the last claim:
$\arg(\bar\varphi/\varphi) = -2\arg\varphi = -2\,\mathrm{Im}\log\varphi$,
and as $\varphi = e^{g}$ ranges over the zero-free holomorphic
functions, $\mathrm{Im}\,g$ ranges over all harmonic functions on the
simply connected $\D$; conversely every harmonic $u$ on $\D$ is
$\mathrm{Im}\,g$ for some holomorphic $g$.
\end{proof}

\begin{remark}[Well-definedness of the minimal form]\label{rem:well-defined}
The reduction of Theorem~\ref{thm:reduction} involves choices --- the
straightening substitution is canonical, but the uniformizing chart,
the Riemann map, and the Pompeiu primitive are not. Any two reductions
of the same equation differ by an admissible morphism between minimal
forms, hence, by Proposition~\ref{prop:residual}, by an element of the
residual groupoid. Every quantity constructed below that is invariant
under \eqref{eq:orbit} is therefore an invariant of the original
equation, well-defined up to the single remaining ambiguity: the
M\"obius transformation relating two choices of chart. Invariants will
accordingly be \emph{fields on $\D$ modulo $\Mob$}, and numbers only
when a further integration kills the M\"obius freedom.
\end{remark}

With the residual symmetries computed and split, the construction of
the invariant pair is a short exercise in hyperbolic geometry, and is
carried out in the next section.

\section{The moduli pair}\label{sec:pair}

The orbit relation \eqref{eq:orbit-split-2} instructs the construction:
whatever is to be invariant must absorb one conformal factor on the
modulus and one harmonic function on the phase. The hyperbolic geometry
of $\D$ supplies exactly these two absorptions --- the first through
the conformal distortion identity of the M\"obius group, the second
through the Laplacian, which annihilates harmonic additions and
transforms conformally. This section builds the resulting pair of
fields and proves their exact equivariance.

\begin{definition}[The pair]\label{def:pair}
Let $w_{\bar z} = B\,\bar w$ be a minimal form on $\D$. Its
\emph{hyperbolic mass density} is the function
\[
\vartheta \;:=\; \tfrac14\,\bigl(1 - |z|^2\bigr)^2\,|B|^2 \;\ge\; 0,
\qquad Z(\vartheta) = Z(B).
\]
The minimal form is \emph{vortex-free} if $B$ is zero-free on $\D$;
then $B$ admits a continuous argument $\theta := \arg B$, single-valued
because $\D$ is simply connected and unique up to an additive constant
in $2\pi\Z$. When moreover some --- equivalently, the branches
differing by constants, every --- branch $\theta \in C^2(\D)$, as
happens in particular when $B \in C^2_{\mathrm{loc}}(\D)$, the
\emph{phase curvature} is the function
\[
\kappa \;:=\; \tfrac14\,\bigl(1 - |z|^2\bigr)^2\,\Delta\theta,
\]
independent of the branch. Both
normalizations are hyperbolic: writing
$\Delta_{\mathrm{hyp}} = \tfrac14(1 - |z|^2)^2\,\Delta$ for the
hyperbolic Laplacian and $\hyparea = 4(1 - |z|^2)^{-2}\,dx\,dy$ for the
hyperbolic area form, one has
$\kappa = \Delta_{\mathrm{hyp}}\arg B$ and, as
Remark~\ref{rem:currents} records, $\vartheta\,\hyparea$ is the mass
2-form of \cite{mass, framed}.
\end{definition}

Two lemmas carry the equivariance; both are classical, and their proofs
are displayed because the whole paper is the pair of cancellations they
encode.

\begin{lemma}[Hyperbolic identity]\label{lem:hyp}
For every $F \in \Mob$ and $z \in \D$,
\[
1 - |F(z)|^2 \;=\; |F'(z)|\,\bigl(1 - |z|^2\bigr).
\]
\end{lemma}

\begin{proof}
Write $F(z) = e^{i\alpha}(z - a)/(1 - \bar a z)$ with $|a| < 1$. Then
\[
1 - |F(z)|^2
= \frac{|1 - \bar a z|^2 - |z - a|^2}{|1 - \bar a z|^2}
= \frac{(1 - |a|^2)(1 - |z|^2)}{|1 - \bar a z|^2},
\qquad
F'(z) = \frac{e^{i\alpha}\,(1 - |a|^2)}{(1 - \bar a z)^2},
\]
the numerator identity by expanding both squares and cancelling the
cross terms $2\operatorname{Re}(\bar a z)$; comparing moduli gives the
claim.
\end{proof}

\begin{lemma}[Conformal transformation of the Laplacian]\label{lem:conf-lap}
Let $F$ be holomorphic on $\D$ and $u \in C^2$. Then
\[
\Delta(u \circ F) \;=\; |F'|^2\,(\Delta u)\circ F .
\]
Moreover, if $F'$ is zero-free --- as for every $F \in \Mob$ --- then
$\arg F'$ is harmonic on $\D$.
\end{lemma}

\begin{proof}
With $\Delta = 4\,\partial\bar\partial$ and $F$ holomorphic,
$\bar\partial(u \circ F) = (\bar\partial u \circ F)\,\overline{F'}$,
and applying $\partial$,
\[
\partial\bar\partial(u \circ F)
= (\partial\bar\partial u \circ F)\,F'\,\overline{F'}
+ (\bar\partial u \circ F)\,\partial\overline{F'}
= |F'|^2\,(\partial\bar\partial u)\circ F,
\]
since $\partial\overline{F'} = \overline{\bar\partial F'} = 0$. For the
last claim, $F'$ is holomorphic and zero-free on the simply connected
$\D$, so $\log F'$ has a holomorphic branch and
$\arg F' = \operatorname{Im}\log F'$ is harmonic.
\end{proof}

\begin{theorem}[Equivariance of the pair]\label{thm:equivariance}
Under the residual groupoid of Proposition~\ref{prop:residual}:
\begin{itemize}
\item[(i)] every holomorphic gauge $B \mapsto (\bar\varphi/\varphi)B$
preserves the vortex-free sector and fixes both fields pointwise:
\[
\vartheta' = \vartheta, \qquad \kappa' = \kappa ;
\]
\item[(ii)] every $F \in \Mob$, acting by
$B \mapsto \overline{F'}\,(B \circ F)$, preserves the vortex-free
sector and moves both fields by composition:
\[
\vartheta' \;=\; \vartheta \circ F, \qquad \kappa' \;=\; \kappa \circ F .
\]
\end{itemize}
Consequently, by Remark~\ref{rem:well-defined}, the pair
$(\vartheta, \kappa)$ --- taken modulo the simultaneous action
$(\vartheta, \kappa) \sim (\vartheta \circ F, \kappa \circ F)$ of
$\Mob$ --- is an invariant of the sourceless equation itself, not
merely of its minimal representative.
\end{theorem}

\begin{proof}
(i) The gauge factor $\bar\varphi/\varphi$ is unimodular and zero-free,
so $|B'| = |B|$, $Z(B') = Z(B)$, and $\vartheta' = \vartheta$. For the
phase, $\theta' = \theta - 2\arg\varphi$ up to an additive constant,
and $\arg\varphi = \operatorname{Im}\log\varphi$ is harmonic
($\varphi$ holomorphic zero-free on the simply connected $\D$); hence
$\Delta\theta' = \Delta\theta$ and $\kappa' = \kappa$.

(ii) The weight $\overline{F'}$ is zero-free, so the vortex-free sector
is preserved. Modulus: $|B'| = |F'|\,|B \circ F|$, and by
Lemma~\ref{lem:hyp},
\[
\vartheta'
= \tfrac14\,(1 - |z|^2)^2\,|F'|^2\,|B \circ F|^2
= \tfrac14\,\bigl(1 - |F|^2\bigr)^2\,|B|^2 \circ F
= \vartheta \circ F .
\]
Phase: continuous branches satisfy
$\theta' = \theta \circ F - \arg F'$ up to an additive constant, so by
Lemma~\ref{lem:conf-lap} --- the harmonicity of $\arg F'$ killing the
second term and the conformal law transforming the first ---
\[
\Delta\theta' = \Delta(\theta \circ F) = |F'|^2\,(\Delta\theta)\circ F,
\qquad
\kappa'
= \tfrac14\,(1 - |z|^2)^2\,|F'|^2\,(\Delta\theta)\circ F
= \tfrac14\,\bigl(1 - |F|^2\bigr)^2\,(\Delta\theta)\circ F
= \kappa \circ F,
\]
the same hyperbolic identity absorbing the conformal factor twice
over. The final statement is Remark~\ref{rem:well-defined}: two
minimal reductions of one equation differ by a residual morphism,
under which, by (i) and (ii), the pair moves by at most a common
M\"obius composition.
\end{proof}

The theorem realizes the division of labor announced in the
introduction: the modulus law of \eqref{eq:orbit-split-2} is absorbed
by the conformal factor of Lemma~\ref{lem:hyp}, the phase law by the
Laplacian of Lemma~\ref{lem:conf-lap}, and nothing of either survives
into the pair except the composition with $F$. A first dividend is
that the second field settles, within the minimal slice, the positive
normalization question of \cite[\S 4]{charge}:

\begin{corollary}[$\kappa$ is the exact obstruction to positivity]\label{cor:positive}
A vortex-free minimal form whose argument admits a $C^2$ branch ---
in particular, one with $B \in C^2_{\mathrm{loc}}$ --- is
holomorphic-gauge equivalent to the positive field $|B|$ if and only
if $\kappa \equiv 0$.
\end{corollary}

\begin{proof}
By Proposition~\ref{prop:weight-gauge} the gauge phases are exactly
the harmonic functions: a gauge carries $\theta$ to $\theta - 2\arg\varphi$
with $-2\arg\varphi$ an arbitrary harmonic function. The phase is
removable --- $\theta - 2\arg\varphi \equiv 0$ for some gauge --- if
and only if $\theta$ is itself harmonic, i.e.\ $\Delta\theta = 0$,
i.e.\ $\kappa \equiv 0$; and a gauge fixes $|B|$, so the normalized
field is $|B|$.
\end{proof}

\begin{remark}[The two currents]\label{rem:currents}
Both fields are densities, against the invariant hyperbolic area, of
2-forms of the covariance type of \cite{mass, framed}. At the minimal
representative ($\mu = 0$, trivial frame, $N = B$ up to sign) the mass
2-form of \cite{mass} is
$\Theta = |B|^2\,dx\,dy = \vartheta\,\hyparea$; and the \emph{phase
curvature form}
\[
K \;:=\; \Delta\theta\; dx\,dy \;=\; \kappa\,\hyparea
\]
is its partner: by Theorem~\ref{thm:equivariance} and the M\"obius
invariance of $\hyparea$ (Lemma~\ref{lem:hyp} squared), both $\Theta$
and $K$ are fixed by gauges and pulled back, $\tilde\Theta = F^*\Theta$
and $\tilde K = F^* K$, by changes of variables. The program of
\cite{mass, framed} produced one covariant 2-form from the modulus of
the numerator field; the phase, which the mass forgets, carries a
second 2-form of the same covariance, and the pair
$(\vartheta, \kappa)$ is the pair of their hyperbolic densities. The
form $K$, unlike $\Theta$, is not sign-definite, and it is the object
that will extend across the vortices in
Section~\ref{sec:vortices} --- where it acquires a companion current
carrying the charge. Throughout, the primary invariant of the phase
is the current $K$; when $K$ admits a density against hyperbolic
area, that density is $\kappa$, and everything the completeness
argument needs from the phase it takes at the current level
(Remark~\ref{rem:distributional}).
\end{remark}

\begin{remark}[Regularity, and the distributional phase curvature]\label{rem:distributional}
The reduction of Theorem~\ref{thm:reduction} delivers only
$B \in C^{\alpha}_{\mathrm{loc}}$, below the $C^2$-branch hypothesis
of Definition~\ref{def:pair}; the gap is closed in either of two ways.
Upward: each step of the reduction loses at most one derivative on the
data it differentiates and the two solves gain one on their output, so
data two degrees smoother than Definition~\ref{def:class} --- for
example a projective representative with
$\nu \in C^{3,\alpha}_{\mathrm{loc}}$ and de-scaled lower-order data
in $C^{2,\alpha}_{\mathrm{loc}}$ --- place $B$ in
$C^{2,\alpha}_{\mathrm{loc}}$ and make $\kappa$ a function. Downward,
and sufficient for everything this paper proves: for merely continuous
zero-free $B$ the phase curvature exists as the distribution
\[
\langle K, \psi \rangle \;:=\; \int_\D \theta\,\Delta\psi\; dx\,dy,
\qquad \psi \in C^\infty_c(\D),
\]
independent of the branch. Gauge additions are harmonic and are killed
weakly --- $\int (\arg\varphi)\,\Delta\psi = \int \Delta(\arg\varphi)\,\psi = 0$
by two integrations by parts against the compactly supported $\psi$ ---
and the M\"obius covariance $\tilde K = F^* K$ holds by the conformal
change of variables in the pairing, Lemma~\ref{lem:conf-lap} applied to
the test function: $\langle \tilde K, \psi \rangle = \langle K, \psi \circ F^{-1} \rangle$.
The completeness argument of the next section uses $K$ only through
Weyl's lemma --- a distributionally harmonic phase difference is
harmonic --- and therefore runs at bare continuity.
\end{remark}

With the pair constructed and its equivariance exact, completeness is
a short argument: matching $\vartheta$ matches the modulus, matching
$K$ makes the phase difference distributionally harmonic, and a
harmonic phase difference is a gauge. The next section carries it out.

\section{Completeness on the vortex-free sector}\label{sec:completeness}

This section proves that the pair separates orbits. The statement
requires knowing that ``vortex-free'' is a property of the equation
and not of the representative; that is settled first, and then the
theorem is a three-line assembly of what Sections~\ref{sec:reduction}
and~\ref{sec:pair} have prepared: matching $\vartheta$ matches the
modulus, matching $K$ makes the phase difference distributionally
harmonic, and a harmonic phase difference is a gauge.

\begin{lemma}[Vortex-freeness is intrinsic]\label{lem:intrinsic}
For a sourceless framed equation per Definition~\ref{def:class} the
following are equivalent: \emph{(i)} its numerator field $N$ is
zero-free on $\Omega$; \emph{(ii)} some minimal reduction is
vortex-free; \emph{(iii)} every minimal reduction is vortex-free.
\end{lemma}

\begin{proof}
Under the three actions the numerator field is multiplied by zero-free
factors and its zero set relocated by homeomorphisms
\cite[Prop.~4.2]{charge}, and a minimal reduction is a finite
composite of admissible morphisms (Theorem~\ref{thm:reduction}); hence
$Z(B)$ is a homeomorphic image of $Z(N)$, and one is empty if and only
if the other is.
\end{proof}

We say the equation itself is \emph{vortex-free} in this case. Recall
from Remark~\ref{rem:distributional} that at the bare regularity of
Definition~\ref{def:class} the phase curvature is read as the current
$K$, with the M\"obius relation $K_1 = F^* K_2$ standing in for
$\kappa_1 = \kappa_2 \circ F$; the two coincide when the fields are
$C^2_{\mathrm{loc}}$.

\begin{theorem}[Completeness]\label{thm:completeness}
Let $\Ecal_1, \Ecal_2$ be vortex-free sourceless framed equations per
Definition~\ref{def:class}, on bounded simply connected domains
$\Omega_1, \Omega_2$, with minimal forms $B_i$ and pairs
$(\vartheta_i, K_i)$ on $\D$, $i = 1, 2$. Then $\Ecal_1$ and
$\Ecal_2$ are equivalent if and only if there is $F \in \Mob$ with
\[
\vartheta_1 \;=\; \vartheta_2 \circ F
\qquad\text{and}\qquad
K_1 \;=\; F^* K_2
\]
--- when the phases admit $C^2$ branches (in particular when
$B_1, B_2 \in C^2_{\mathrm{loc}}$), the second condition reading
$\kappa_1 = \kappa_2 \circ F$.
\end{theorem}

\begin{proof}
Each $\Ecal_i$ is equivalent to its minimal form
(Theorem~\ref{thm:reduction}), so it suffices to prove the claim for
$B_1, B_2$.

\emph{Necessity.} If $B_1$ and $B_2$ are equivalent, they are related
by a residual morphism (Proposition~\ref{prop:residual}), and every
residual morphism is a single gauge followed by a single M\"obius map:
gauges and M\"obius maps each form a group, and M\"obius maps
normalize the gauges --- if $\varphi$ is a holomorphic zero-free gauge
and $F \in \Mob$, then $\varphi \circ F$ is again one --- so any
alternating composite collapses. By
Theorem~\ref{thm:equivariance} and Remark~\ref{rem:distributional},
the gauge fixes the pair and the M\"obius map $F$ carries it by
composition and pullback: $\vartheta_1 = \vartheta_2 \circ F$,
$K_1 = F^* K_2$.

\emph{Sufficiency.} Let $F \in \Mob$ realize the relation, and set
$\tilde B_2 := \overline{F'}\,(B_2 \circ F)$, the M\"obius pullback of
$B_2$ --- a minimal form equivalent to $B_2$
(Proposition~\ref{prop:residual}(ii)), vortex-free, with pair
$(\vartheta_2 \circ F,\; F^* K_2) = (\vartheta_1, K_1)$.

\emph{Moduli match.} From
$\tfrac14(1 - |z|^2)^2\,|\tilde B_2|^2 = \tfrac14(1 - |z|^2)^2\,|B_1|^2$
and $1 - |z|^2 > 0$ on $\D$, taking square roots gives
$|\tilde B_2| = |B_1|$ pointwise.

\emph{Phases differ harmonically.} Choose continuous branches
$\theta_1 = \arg B_1$ and $\tilde\theta_2 = \arg \tilde B_2$ and set
$h := \theta_1 - \tilde\theta_2 \in C^0(\D; \R)$. For every
$\psi \in C^\infty_c(\D)$,
\[
\int_\D h\,\Delta\psi\; dx\,dy
\;=\; \langle K_1, \psi \rangle - \langle K_{\tilde B_2}, \psi \rangle
\;=\; 0,
\]
so $h$ is distributionally harmonic, and by Weyl's lemma harmonic.

\emph{A harmonic phase difference is a gauge.} On the simply connected
$\D$ there is a holomorphic $g$ with $\operatorname{Im} g = -h/2$; the
gauge $\varphi := e^{g}$, holomorphic and zero-free, carries
$\tilde B_2$ to a minimal form of the same modulus $|B_1|$ and phase
\[
\tilde\theta_2 - 2\arg\varphi
\;=\; \tilde\theta_2 - 2\operatorname{Im} g
\;=\; \tilde\theta_2 + h
\;=\; \theta_1,
\]
after absorbing the additive $2\pi\Z$ constant of the branches into
$h$. The gauged field equals $B_1$, and the chain
$\Ecal_1 \sim B_1 = (\text{gauge})\,\tilde B_2 \sim B_2 \sim \Ecal_2$ closes.
\end{proof}

The converse direction of the moduli description --- which pairs
actually occur --- is free, and freeness is itself a statement: no
compatibility condition couples the two fields.

\begin{proposition}[Realization, and the range at bare regularity]\label{prop:realization}
\begin{itemize}
\item[(i)] For every $\vartheta \in C^{\alpha}_{\mathrm{loc}}(\D)$ with
$\vartheta > 0$ pointwise and every
$\kappa \in C^{\alpha}_{\mathrm{loc}}(\D; \R)$ there is a vortex-free
minimal form --- a member of the class of
Definition~\ref{def:class}, on $\Omega = \D$ --- with hyperbolic mass
density $\vartheta$ and phase curvature $\kappa$; and by
Theorem~\ref{thm:completeness} it is unique up to the residual
groupoid.
\item[(ii)] At the bare regularity of the class, the pairs
$(\vartheta, K)$ of vortex-free minimal forms are exactly those with
\[
\vartheta \in C^{\alpha}_{\mathrm{loc}}(\D),\quad \vartheta > 0,
\qquad
K = \Delta\theta\,dx\,dy \ \text{ in } \mathcal{D}'(\D)\ \text{ for
some } \theta \in C^{\alpha}_{\mathrm{loc}}(\D; \R).
\]
\end{itemize}
\end{proposition}

\begin{proof}
(i) The modulus is forced: set
$|B| := 2\sqrt{\vartheta}\,(1 - |z|^2)^{-1} \in C^{\alpha}_{\mathrm{loc}}(\D)$,
positive. For the phase, solve
\[
\Delta\theta \;=\; \frac{4\,\kappa}{(1 - |z|^2)^2} \;\in\; C^{\alpha}_{\mathrm{loc}}(\D):
\]
the right side need not be integrable up to the boundary, but the
Laplacian, elliptic with constant coefficients, is surjective on
$\mathcal{D}'(\D)$ \cite[Ch.~10]{hormander}, and elliptic regularity
lifts the solution to $\theta \in C^{2,\alpha}_{\mathrm{loc}}(\D; \R)$.
Then $B := |B|\, e^{i\theta} \in C^{\alpha}_{\mathrm{loc}}(\D)$ is
zero-free, the minimal form $w_{\bar z} = B\bar w$ has the data of
Definition~\ref{def:class} --- frame $(1, 0)$, $\mu = 0$, $\fa = 0$,
$\fb = -B$ --- and its pair is $(\vartheta, \kappa)$ by construction.

(ii) For a zero-free $B \in C^{\alpha}_{\mathrm{loc}}(\D)$, a
continuous branch $\theta = \arg B$ is locally the imaginary part of
a $C^{\alpha}$ branch of $\log B$ --- $B$ being locally bounded away
from zero --- so $\theta \in C^{\alpha}_{\mathrm{loc}}(\D; \R)$, and
the phase-curvature current of Remark~\ref{rem:distributional} is, by
its defining pairing, $\Delta\theta$ in $\mathcal{D}'(\D)$.
Conversely, given such $(\vartheta, \theta)$, the field
$B := 2\sqrt{\vartheta}\,(1 - |z|^2)^{-1}\,e^{i\theta}
\in C^{\alpha}_{\mathrm{loc}}(\D)$ is a vortex-free minimal form with
density $\vartheta$ and current $K = \Delta\theta\,dx\,dy$.
\end{proof}

\begin{corollary}[The moduli space]\label{cor:moduli}
The assignment $\Ecal \mapsto [(\vartheta, K)]$ induces a bijection
\begin{multline*}
\bigl\{\text{vortex-free sourceless framed equations}\bigr\}/\!\sim
\\
\;\longrightarrow\;
\Bigl\{\,(\vartheta,\, \Delta\theta\,dx\,dy) \;:\; \vartheta, \theta \in C^{\alpha}_{\mathrm{loc}}(\D;\R),\ \vartheta > 0\,\Bigr\} \Big/ \Mob\,,
\end{multline*}
which restricts, on the equations whose phase curvature is a
$C^{\alpha}_{\mathrm{loc}}$ function, to a bijection onto
\[
\Bigl\{\,(\vartheta, \kappa) \;:\; \vartheta, \kappa \in C^{\alpha}_{\mathrm{loc}}(\D;\R),\ \vartheta > 0\,\Bigr\} \Big/ \Mob\,,
\]
the M\"obius group acting by composition on the functions and by
pullback on the current --- an action preserving each displayed set,
since $F^*(\Delta\theta\,dx\,dy) = \Delta(\theta \circ F)\,dx\,dy$ by
conformal covariance. The moduli of the vortex-free sourceless class
are a positive H\"older density and the Laplacian current of a
H\"older phase, both unconstrained, modulo a three-parameter group
--- on the function-valued sector, two unconstrained scalar fields.
\end{corollary}

\begin{proof}
Well-definedness and injectivity are Theorems~\ref{thm:equivariance}
and~\ref{thm:completeness}; surjectivity onto the two displayed sets
is Proposition~\ref{prop:realization}, parts (ii) and (i)
respectively.
\end{proof}

\begin{remark}[Relation to Bers' classification of generating pairs]\label{rem:bers}
Theorem~\ref{thm:completeness} is not Bers' theory repackaged, and it
is worth saying exactly why. On the pure-frame slice the frame is a
generating pair and the numerator field is, up to the frame
determinant, Bers' characteristic coefficient $b_{(F,G)}$
\cite[\S 11]{framed}; Bers' own equivalences of pairs
\cite{bers} operate over a \emph{fixed} conformal structure on a
fixed domain, and classify pairs through their characteristic
coefficients. The equivalence of the present paper is the full
groupoid --- every recombination of the unknown, every scaling, and
every orientation-preserving change of variables, the M\"obius action
included --- and the theorem's content is not the choice of the
invariant but the \emph{computation of the quotient}: the
identification of exactly which combinations of $|B|$ and $\arg B$
--- the hyperbolic density and the phase curvature, and nothing more
--- survive the whole group. That the quotient is an explicitly
identified infinite-dimensional function space modulo a
finite-dimensional group, rather than something wild, is the
substance; to the author's knowledge, no part of the classical theory
computes it. Section~\ref{sec:solutions}
completes the comparison at the level of solutions, where Bers'
equipotent pairs appear as the fiber upstream of the present quotient.
\end{remark}

\begin{remark}[Where simple connectivity is spent]\label{rem:simply-connected}
The hypothesis enters the sufficiency proof exactly once, at its last
step: a harmonic function on $\D$ is the imaginary part of a
holomorphic function, i.e.\ a gauge phase
(Proposition~\ref{prop:weight-gauge}). On a multiply connected domain
the phase difference $h$ is still harmonic, but only its
\emph{period-free} part is a gauge phase: the periods of the conjugate
differential $\star dh$ around the holes survive as additional
invariants, and the mod-$2$ phenomena of \cite[Rem.~3.6]{charge}
resurface inside them. We return to this in
Section~\ref{sec:discussion}; everything else in the proof ---
reduction, equivariance, Weyl --- is insensitive to the topology of
$\Omega$.
\end{remark}

\begin{remark}[The deficit of the mass, measured]\label{rem:deficit}
On the vortex-free sector the charge vanishes identically --- a
zero-free continuous field on a simply connected domain cannot wind
\cite[\S 3]{charge} --- so there the pair $(\Mcal, n)$ reduces to the
mass alone, and Corollary~\ref{cor:moduli} measures its deficit
exactly: of the two unconstrained fields modulo $\Mob$, the mass
retains the single number $\int_\D \vartheta\,\hyparea$. The ledger
of what each invariant retains and forgets is closed in
Remark~\ref{rem:two-numbers}, and
Section~\ref{sec:incompleteness} converts it into explicit pairs of
inequivalent equations that mass and charge cannot separate.
\end{remark}

\section{Mass and charge as projections of the moduli}\label{sec:massandcharge}

The mass, and the numerical charge wherever it is charge-admissible,
are invariants of the equivalence class and hence functionals of the
complete moduli data --- the pair $(\vartheta, K)$ on the vortex-free
sector, the triple $(\vartheta,\, d\eta,\, d{\star}\eta)$ of
Section~\ref{sec:vortices} across the vortices. This section identifies
the relevant functionals --- one moment of the first field and,
whenever defined, one total flux carried by the charge current,
purely atomic on the tame sector and absent from the vortex-free pair
--- and displays the much larger family of invariants generated by
the moduli.

\begin{proposition}[The mass is the zeroth hyperbolic moment]\label{prop:mass-moment}
For every sourceless framed equation per Definition~\ref{def:class},
with minimal hyperbolic density $\vartheta$,
\[
\Mcal \;=\; \int_\D \vartheta\; \hyparea \;\in\; [0, \infty] .
\]
On the vortex-free sector $\Mcal > 0$; it may be infinite, the class
imposing no integrability at the boundary.
\end{proposition}

\begin{proof}
The mass is invariant under every admissible morphism
\cite[Thm.~8.2]{framed}, hence computable on the minimal
representative, where the mass 2-form is
$\Theta = |\fb|^2 (1 - |\mu|^2)^{-1}\, dx\,dy = |B|^2\, dx\,dy
= \vartheta\,\hyparea$ (Remark~\ref{rem:currents}). The value is
independent of the choice of reduction, as it must be: two reductions
differ by a residual morphism, gauges fix $\vartheta$, and for
$F \in \Mob$ the substitution $\vartheta \mapsto \vartheta \circ F$ is
absorbed by the M\"obius invariance of $\hyparea$
(Lemma~\ref{lem:hyp} squared) under the change of variables.
Positivity on the vortex-free sector is the positivity of the
continuous $\vartheta$.
\end{proof}

The same mechanism --- equivariant integrand against invariant area
--- manufactures invariants wholesale, and it is worth recording how
much the pair generates beyond the mass.

\begin{proposition}[The hyperbolic distribution profile]\label{prop:spectrum}
Let $\Ecal$ be a sourceless framed equation with minimal pair
$(\vartheta, \kappa)$.
\begin{itemize}
\item[(i)] For every Borel $\Phi : [0, \infty) \to [0, \infty]$,
\[
I_\Phi(\Ecal) \;:=\; \int_\D \Phi(\vartheta)\;\hyparea \;\in\; [0, \infty]
\]
is an invariant of the equivalence class. Equivalently, the
\emph{hyperbolic distribution function}
$\lambda_\vartheta(t) := \hyparea\bigl(\{\vartheta > t\}\bigr)$,
$t > 0$, is an invariant function of $t$, and the mass is its layer-cake
integral $\Mcal = \int_0^\infty \lambda_\vartheta(t)\, dt$.
\item[(ii)] On the vortex-free sector with
$B \in C^2_{\mathrm{loc}}$, the joint pushforward measure
$(\vartheta, \kappa)_*\,\hyparea$ on $(0, \infty) \times \R$ is an
invariant. More generally, every M\"obius-equivariant differential
expression of the pair, whenever defined with sufficient regularity,
gives an invariant after integration against $\hyparea$ (possibly
with value $+\infty$). For example, if
$B \in C^3_{\mathrm{loc}}$, this applies to the hyperbolic gradient
energies
$\tfrac14(1 - |z|^2)^2\,|\nabla\vartheta|^2$ and
$\tfrac14(1 - |z|^2)^2\,|\nabla\kappa|^2$.
\end{itemize}
\end{proposition}

\begin{proof}
Gauges fix $\vartheta$ and $\kappa$ pointwise
(Theorem~\ref{thm:equivariance}(i)). For $F \in \Mob$, the fields move
by composition, so $\{\vartheta' > t\} = F^{-1}(\{\vartheta > t\})$
and $\Phi(\vartheta') = \Phi(\vartheta) \circ F$; the change of
variables and $F^*\hyparea = \hyparea$ give the invariance of
$\lambda_\vartheta$, of $I_\Phi$, and of the joint pushforward. The
layer-cake identity is Fubini on
$\{(z, t) : 0 < t < \vartheta(z)\}$. Under the additional regularity
required for the gradient examples,
$|\nabla(u \circ F)|^2 = |F'|^2\,(|\nabla u|^2 \circ F)$ for conformal
$F$, and the factor $|F'|^2$ is absorbed by the hyperbolic weight
exactly as in the proof of Theorem~\ref{thm:equivariance}(ii).
\end{proof}

\begin{remark}[The profile is not complete]\label{rem:spectrum}
The family $I_\Phi$ retains precisely the hyperbolic statistics of
$\vartheta$ and forgets the arrangement: two densities equimeasurable
against $\hyparea$ but not M\"obius-related share every $I_\Phi$ ---
two equal bumps at different hyperbolic separations already do it,
M\"obius maps being hyperbolic isometries. Completeness resides in the
pair of \emph{fields} modulo $\Mob$ (Theorem~\ref{thm:completeness}),
not in any list of numbers; the functionals of
Proposition~\ref{prop:spectrum} are its shadows, of controlled shape,
of which the mass is the crudest.
\end{remark}

\medskip\noindent\textbf{The charge.} The phase component of the
moduli has a current-valued part of an entirely different type from a
moment of a function. On the vortex sector, where $\arg B$ ceases to
be single-valued, the single-valued survivor is the phase 1-form
$\eta = d\arg B$, and its derivative splits into two currents: the
exterior derivative $d\eta$, the \emph{charge current} --- on the tame
sector a purely atomic measure with mass $2\pi n_p$ at each vortex ---
and the co-derivative $d{\star}\eta$, which extends $K$ across the
vortices. On the tame sector Theorem~\ref{thm:atoms} proves that the
latter has no atomic part; in the general phase-integrable sector both
objects are currents and may have singular components. Whenever the
numerical pseudo-analytic charge is defined --- in particular when the
zero set is compactly contained --- it is the normalized total flux
of $d\eta$; on the tame sector this is equivalently the total mass of
the atomic charge measure. The second current is the extended phase
curvature. On the vortex-free sector $\eta$ is exact,
$d\eta = 0$, and $d{\star}\eta = K$, consistently with the numerical
charge, when considered, vanishing there
(Remark~\ref{rem:deficit}).

\begin{remark}[Mass and charge as numerical projections]\label{rem:two-numbers}
The ledger closes. Of the complete moduli data, the mass retains one number
--- the zeroth moment of $\vartheta$ --- and is blind to every higher
moment, to the arrangement behind the profile, and to the whole of
$K$; whenever the numerical charge is defined, it retains the total
flux of $d\eta$ --- the sum of its flux atoms on the tame sector ---
and is blind to the whole of the phase curvature and of $\vartheta$.
The independence of mass and charge \cite[Thm.~6.1]{charge} is now structural rather than
constructed: the two numbers are numerical projections of distinct
components of the moduli --- on the vortex-free sector
Proposition~\ref{prop:realization} shows the pair
$(\vartheta, \kappa)$ to be unconstrained, and
Corollary~\ref{cor:fibers} realizes every value $(m, k)$ of the two
numbers jointly. And the fibers of $(\Mcal, n)$ are as large as
Corollary~\ref{cor:moduli} says: over a given mass, the entire
codimension-one level set $\{\int_\D \vartheta\,\hyparea = \Mcal\}$ of one
field, times the entirety of the other, modulo three parameters. The
next section converts this bookkeeping into explicit pairs of
inequivalent equations that mass and charge cannot separate.
\end{remark}

\section{Incompleteness of mass and charge}\label{sec:incompleteness}

The bookkeeping of Remark~\ref{rem:two-numbers} becomes concrete here:
pairs of inequivalent equations with equal mass and equal charge,
separated along each axis of the moduli independently, and a corollary
sizing the fibers of $(\Mcal, n)$. Each example is a pair of minimal
forms $w_{\bar z} = B\bar w$ on $\D$ --- members of the class of
Definition~\ref{def:class} in their own right --- and, by
Theorem~\ref{thm:reduction}, stands for the pair of full equivalence
classes of framed equations reducing to them.

\begin{example}[Equal mass, equal charge, distinct profile]\label{ex:spectrum}
Fix $t > 0$ and take
\[
B_1 \;=\; t, \qquad
B_2 \;=\; \sqrt{2}\,t\,\bigl(1 - |z|^2\bigr)^{1/2} .
\]
Both fields are smooth, positive, and zero-free on $\D$: both
equations are vortex-free, of charge $0$
(Remark~\ref{rem:deficit}), and of identically vanishing phase
curvature --- $\arg B_i \equiv 0$ is harmonic, so
$\kappa_1 = \kappa_2 \equiv 0$ and the entire phase data of the two
equations coincide. The masses coincide as well:
\[
\Mcal_1 = \int_\D t^2\, dx\,dy = \pi t^2,
\qquad
\Mcal_2 = \int_\D 2t^2(1 - |z|^2)\, dx\,dy = \pi t^2 .
\]
But the second hyperbolic moment, the invariant
$I_{\Phi}$ of Proposition~\ref{prop:spectrum} at $\Phi(s) = s^2$,
separates them:
\[
I_{s^2}(B_1)
= \frac{1}{4}\int_\D t^4\,(1 - |z|^2)^2\, dx\,dy
= \frac{\pi t^4}{12},
\qquad
I_{s^2}(B_2)
= \int_\D t^4\,(1 - |z|^2)^4\, dx\,dy
= \frac{\pi t^4}{5} .
\]
The two equations are therefore inequivalent, and nothing but the
\emph{shape} of $\vartheta$ --- the same total hyperbolic mass,
distributed differently --- tells them apart: $\vartheta_1$ is
$\tfrac14 t^2(1-|z|^2)^2$ and $\vartheta_2$ is
$\tfrac12 t^2(1-|z|^2)^3$, equal in integral, distinct in profile, and
no M\"obius map carries one profile to the other because the moment
invariant forbids it without any case analysis.
\end{example}

\begin{example}[Identical $\vartheta$, distinct phase curvature]\label{ex:phase}
Take
\[
B_1 \;=\; 1, \qquad B_2 \;=\; e^{\,i x^2}, \qquad z = x + iy .
\]
Both fields are unimodular: $\vartheta_1 = \vartheta_2 =
\tfrac14(1 - |z|^2)^2$ \emph{identically}, so the two equations share
the mass $\Mcal = \pi$, the entire hyperbolic distribution profile --- every
invariant $I_\Phi$ of Proposition~\ref{prop:spectrum}, every moment,
every level-set measure --- and, being vortex-free, the charge $0$.
Every metric invariant of this paper and of \cite{mass, framed} is
blind to the pair. The phase curvature is not:
$\theta_1 \equiv 0$ gives $\kappa_1 \equiv 0$, while
$\theta_2 = x^2$ gives $\Delta\theta_2 = 2$ and
\[
\kappa_2 \;=\; \tfrac12\,\bigl(1 - |z|^2\bigr)^2 \;>\; 0
\quad\text{on all of } \D .
\]
If the equations were equivalent, the necessity half of
Theorem~\ref{thm:completeness} would produce $F \in \Mob$ with
$\kappa_1 = \kappa_2 \circ F$ --- a function vanishing identically
equal to one vanishing nowhere. They are inequivalent, separated by
the phase alone.
\end{example}

\begin{remark}[The two axes]\label{rem:two-axes}
The examples are dual. Example~\ref{ex:spectrum} freezes the phase
data entirely ($\kappa \equiv 0$ for both) and moves only the shape of
$\vartheta$; Example~\ref{ex:phase} freezes $\vartheta$ entirely and
moves only $\kappa$. Together they realize, at the level of
counterexamples, the independence of the two fields that
Proposition~\ref{prop:realization} established at the level of the
moduli: the pair $(\Mcal, n)$ fails along each axis separately, not
merely along some diagonal. The independence theorem of
\cite[Thm.~6.1]{charge} scaled one number against the other; here each
\emph{field} is scaled against the other's invariants.
\end{remark}

\begin{example}[The vortex set itself]\label{ex:vortex-set}
The pair $(\Mcal, n)$ does not even register the crudest invariant
below the moduli. Let $h : [0, 1) \to [0, \infty)$ be smooth,
vanishing on $[0, \tfrac12]$ and positive on $(\tfrac12, 1)$, and set
$B_2(z) := c\,h(|z|)$ with $c > 0$ arranged so that
$\int_\D B_2^2 = \pi t^2$; take $B_1 = t$ as before. Both equations
have mass $\pi t^2$; both are admissible for the charge, with
$Z(B_2) = \overline{D(0, \tfrac12)} \Subset \D$, and both have charge
$0$ --- $B_2$ is real and nonnegative near $\partial\D$, so its
argument cannot wind. But $Z(B_1) = \emptyset$ and $Z(B_2)$ is a
closed disk, and the vortex set is an invariant of the class up to the
homeomorphisms of the class itself \cite[Prop.~4.2]{charge}: the
equations are inequivalent. Here neither field-invariant is needed;
$(\Mcal, n)$ fails before the moduli theory even begins.
\end{example}

\begin{corollary}[The fibers of $(\Mcal, n)$]\label{cor:fibers}
For every $(m, k) \in (0, \infty) \times \Z$, the fiber of
$(\Mcal, n)$ over $(m, k)$ contains an infinite-dimensional family of
pairwise inequivalent sourceless framed equations.
\end{corollary}

\begin{proof}
\emph{Case $k = 0$.} Fix $t$ with $\pi t^2 = m$ and let
$\vartheta_0 := \tfrac14 t^2 (1 - |z|^2)^2$ be the density of
$B_1 = t$: radial, with unique maximum at $0$, so any $F \in \Mob$
with $\vartheta_0 \circ F = \vartheta_0$ fixes $0$ and is a rotation
--- and every rotation qualifies. For real
$\sigma \in C^{\infty}_c(\D)$ let $\Ecal_\sigma$ be the minimal form
$B_\sigma := t\,e^{i\sigma}$ --- smooth, a member of the class of
Definition~\ref{def:class} in its own right --- whose modulus is $t$
and whose phase curvature is
$\tfrac14(1-|z|^2)^2\Delta\sigma$, everything classical since $\sigma$
is smooth. Every $\Ecal_\sigma$ is vortex-free with mass
$m$ and charge $0$. By Theorem~\ref{thm:completeness},
$\Ecal_{\sigma_1} \sim \Ecal_{\sigma_2}$ forces
$\vartheta_0 = \vartheta_0 \circ F$ and hence $F$ a rotation, and then
$\Delta\sigma_1 = (\Delta\sigma_2) \circ F$. Since $\Delta$ is
injective on $C^\infty_c$ (a compactly supported harmonic function
vanishes). To obtain a literal injection, restrict to the
infinite-dimensional subspace of radial
$\sigma \in C^\infty_c(\D)$: rotations then fix both $\sigma$ and
$\Delta\sigma$, so equivalence forces $\sigma_1=\sigma_2$. Thus the
fiber contains an infinite-dimensional pairwise inequivalent
subfamily.

\emph{Case $k \neq 0$.} Take
$B_\sigma := t\,z^{k}\,e^{i\sigma}$ for $k > 0$ (and
$t\,\bar z^{\,|k|} e^{i\sigma}$ for $k < 0$), with real
$\sigma \in C^{\infty}_c(A)$ supported in the annulus
$A = \{\tfrac13 < |z| < \tfrac23\}$ and $t$ fixed by the mass, which
the unimodular factor $e^{i\sigma}$ does not alter. Each $B_\sigma$
has the single vortex $0$ of local charge $k$ --- the compactly
supported single-valued $\sigma$ moves no winding --- hence charge
$k$. Suppose $\Ecal_{\sigma_1} \sim \Ecal_{\sigma_2}$: the residual relation
\eqref{eq:orbit} sends vortex set to vortex set, so $F(0) = 0$ and $F$
is a rotation. On $\D \setminus \{0\}$ every identity in the proof of
Theorem~\ref{thm:equivariance} is local --- it needs only a local
continuous branch of the argument --- so the phase curvature, as a
field on the punctured disk where
$\arg B_\sigma = k\arg z + \sigma$ locally and $k \arg z$ is locally
harmonic, satisfies
$\tfrac14(1-|z|^2)^2\Delta\sigma_1 = \bigl(\tfrac14(1-|z|^2)^2\Delta\sigma_2\bigr) \circ F$
there, everything again classical for smooth $\sigma$. As before, restrict to the infinite-dimensional subspace of radial
$\sigma \in C^\infty_c(A)$. Rotations act trivially there, and
injectivity of $\Delta$ on $C^\infty_c(A)$ makes
$\sigma \mapsto [\Ecal_\sigma]$ injective. Hence the fiber again
contains an infinite-dimensional pairwise inequivalent subfamily.
\end{proof}

The three mechanisms of this section --- the vortex set, the
hyperbolic distribution profile, the phase curvature --- are ordered by how much
of the moduli they read, and the first already predates this paper
\cite{charge}. Each defeats $(\Mcal, n)$ on its own. What none of them
individually achieves is separation of everything, which is
Theorem~\ref{thm:completeness}'s job off the vortices; the next
section extends it across them, as far as the phase
form remains integrable, and locates the charge's irreducible place
in the ledger.

\section{The vortex sector}\label{sec:vortices}

At a vortex the argument of $B$ ceases to be single-valued, and with
it the phase curvature of Definition~\ref{def:pair}. What survives,
single-valued, is the \emph{differential} of the argument, and the
correct objects of the vortex sector are its two exterior
derivatives: one carries the charge and, on the tame sector, is purely
atomic; the other extends the curvature and, again on the tame sector,
has no atoms. In the general phase-integrable sector both are currents
whose singular parts need not be atomic. This section constructs
both currents, proves the split, and extends the completeness theorem
across the vortices --- to the exact frontier where the phase form
remains integrable. The sufficiency half is a distributional
div--curl argument, a planar case of Hodge theory for currents
\cite{derham}: a closed and co-closed locally integrable 1-form is
smooth and harmonic, and on the disk a harmonic 1-form is the
differential of a gauge phase.

\begin{definition}[Phase form; integrability; tameness]\label{def:phase-form}
Let $B$ be a minimal field with $Z(B)$ closed, Lebesgue-null, and
with connected complement, and let $B \in C^1(\D \setminus Z(B))$. The
\emph{phase form} of $B$ is the real 1-form
\[
\eta \;:=\; \operatorname{Im}\frac{dB}{B}
\qquad \text{on } \D \setminus Z(B),
\]
manifestly single-valued, and equal on each subdomain admitting a
continuous branch to $d\arg B$; in particular $\eta$ is closed on
$\D \setminus Z(B)$. The field is \emph{phase-integrable} if
$\eta \in L^1_{\mathrm{loc}}(\D)$; the two currents of the phase are
then
\[
\langle d\eta, \psi \rangle \;:=\; -\int_\D d\psi \wedge \eta,
\qquad
\langle d{\star}\eta, \psi \rangle \;:=\; -\int_\D \langle \eta, d\psi \rangle\; dx\,dy
\;=\; -\int_\D (P\psi_x + Q\psi_y)\; dx\,dy,
\]
for $\psi \in C^\infty_c(\D)$, where $\eta = P\,dx + Q\,dy$ and
$\star dx = dy$, $\star dy = -dx$. On the vortex-free $C^2$ sector,
$\eta = d\theta$ is exact, $d\eta = 0$, and
$d{\star}\eta = \Delta\theta\,dx\,dy = K$ recovers
Remark~\ref{rem:currents}. The field is \emph{tame} if $Z(B)$ is
finite and near each $p \in Z(B)$
\[
B \;=\; \tau_p\; g_p, \qquad
\tau_p(z) :=
\begin{cases}
(z - p)^{n_p}, & n_p \ge 1,\\
(\bar z - \bar p)^{|n_p|}, & n_p \le -1,
\end{cases}
\]
with $g_p$ zero-free and $C^1$; the integer $n_p$ is then the local
charge of $p$ in the sense of \cite{charge}, the winding of $B$ along
a small circle being $n_p$ plus the winding of the zero-free $g_p$,
which vanishes. Tame fields are phase-integrable:
$\eta = n_p\, d\varphi_p + \operatorname{Im}(dg_p/g_p)$ near $p$,
where $\varphi_p := \arg(z - p)$, so $|\eta| = O(|z - p|^{-1})$.
\end{definition}

Both pairings agree with the classical objects on smooth data: for
$\eta = du$ with $u \in C^2$, Stokes gives
$\langle d\eta, \psi \rangle = 0$ and integration by parts gives
$\langle d{\star}\eta, \psi \rangle = \int (\Delta u)\,\psi$. A word
on the name: $d{\star}\eta$ is a 2-form-valued current --- the
divergence current of the phase form --- related to but not literally
the codifferential $\delta\eta$, which is a function; we write
$d{\star}\eta$ throughout because the definition above is the object
actually used, and call it the co-derivative only in this displayed
sense. As with
the $C^2$ of Definition~\ref{def:pair}, the $C^1$ regularity off the
zero set exceeds the bare class of Definition~\ref{def:class} and is
supplied by smoother data, per Remark~\ref{rem:distributional}. The
covariance of the currents costs nothing and holds at full
phase-integrable generality.

\begin{lemma}[Covariance of the currents]\label{lem:covariance}
Phase-integrability is preserved by the residual groupoid, and:
\begin{itemize}
\item[(i)] under a holomorphic gauge, $\eta' = \eta - 2\,d(\arg\varphi)$,
and both currents are unchanged;
\item[(ii)] under $F \in \Mob$, $\tilde\eta = F^*\eta - d(\arg F')$,
and both currents pull back:
\[
d\tilde\eta \;=\; F^*(d\eta), \qquad
d{\star}\tilde\eta \;=\; F^*(d{\star}\eta).
\]
\end{itemize}
\end{lemma}

\begin{proof}
(i) $B' = (\bar\varphi/\varphi)B$ gives
$dB'/B' = dB/B + d\log\bar\varphi - d\log\varphi$, whose imaginary
part is $\eta - 2\,d(\arg\varphi)$; the shift is a smooth exact form
with harmonic potential, so its $d$-current vanishes by Stokes and its
$d{\star}$-current is $\Delta(\arg\varphi)\,dx\,dy = 0$
(Lemma~\ref{lem:conf-lap}). (ii) $\tilde B = \overline{F'}\,(B \circ F)$
gives
$d\tilde B/\tilde B = F^*(dB/B) + d\log\overline{F'}$, whose imaginary
part is $F^*\eta - d(\arg F')$, again an $L^1_{\mathrm{loc}}$ form
plus a smooth exact form with harmonic potential. The exterior
derivative of currents commutes with the diffeomorphism pullback, so
$d\tilde\eta = F^*(d\eta)$; and since the Hodge star on 1-forms in two
dimensions is invariant under orientation-preserving conformal maps,
${\star}F^*\eta = F^*({\star}\eta)$, whence
$d{\star}\tilde\eta = F^*(d{\star}\eta) - \Delta(\arg F')\,dx\,dy
= F^*(d{\star}\eta)$.
\end{proof}

\begin{theorem}[The two curvatures of the phase]\label{thm:atoms}
Let $B$ be tame. Then, as currents on $\D$:
\begin{itemize}
\item[(i)] the exterior derivative is purely atomic, with the local
charges as its fluxes:
\[
d\eta \;=\; 2\pi \sum_{p \in Z(B)} n_p\; \delta_p\,;
\]
the \emph{charge measure} $C := \tfrac{1}{2\pi}\,d\eta$ has total mass
the pseudo-analytic charge, $C(\D) = \sum_p n_p = n$;
\item[(ii)] the co-derivative carries no atoms: near each
$p$, $\;d{\star}\eta = d{\star}\operatorname{Im}(dg_p/g_p)$ is the
distributional Laplacian of the single-valued continuous potential
$\arg g_p$, so that
$\langle d{\star}\eta,\; \psi((\cdot - p)/\varepsilon) \rangle \to 0$
as $\varepsilon \to 0$ for every $\psi \in C^\infty_c$ --- the
concentration test that in (i) detects the atoms --- and away from
$Z(B)$ it is $\Delta\theta\,dx\,dy$. If moreover
$B \in C^2(\D \setminus Z(B))$ and every $g_p \in C^2$ on its full
neighborhood of $p$ --- the first hypothesis making $\Delta\theta$
classical off the vortices, the second continuous across them ---
then
$K := d{\star}\eta$ is a continuous 2-form on all of $\D$ and
$\kappa := \tfrac14(1-|z|^2)^2\,\Delta\theta$ extends continuously
across the vortices.
\end{itemize}
In particular the two currents are mutually blind: $d\eta$ reads
exactly the local charges, $d{\star}\eta$ exactly the curvature, and
neither sees the other.
\end{theorem}

\begin{proof}
Write $\zeta := \eta - \sum_p n_p\,d\varphi_p$ on
$\D \setminus Z(B)$: near each $p$,
$\zeta = \operatorname{Im}(dg_p/g_p) - \sum_{q \neq p} n_q\,d\varphi_q$,
the first term continuous ($g_p$ zero-free $C^1$) and the finite sum
smooth near $p$ --- each $\varphi_q$, $q \neq p$, being there a
smooth harmonic potential --- so $\zeta$ extends continuously across
$Z(B)$, and $\zeta$ is closed on $\D \setminus Z(B)$.

\emph{The model vortex.} In polar coordinates $(\rho, \alpha)$
centered at $p$, $d\varphi_p = d\alpha$, and for
$\psi \in C^\infty_c(\D)$, extended by zero to $\C$,
\[
\langle d(d\varphi_p), \psi \rangle
= -\int d\psi \wedge d\alpha
= -\int_0^{2\pi}\!\!\int_0^\infty (\partial_\rho\psi)\; d\rho\, d\alpha
= 2\pi\,\psi(p),
\]
since $d\psi \wedge d\alpha = (\partial_\rho\psi)\,d\rho \wedge d\alpha$
and each radial integral telescopes to $-\psi(p)$; while
\[
\langle d{\star}(d\varphi_p), \psi \rangle
= -\int \nabla\varphi_p \cdot \nabla\psi
= -\int_0^\infty\!\!\int_0^{2\pi} \frac{\partial_\alpha \psi}{\rho^2}\;\rho\, d\alpha\, d\rho
= 0,
\]
each angular integral vanishing for the single-valued $\psi$. The
model vortex thus deposits its entire content --- the atom
$2\pi\,\delta_p$ --- in the exterior derivative, and nothing in the
co-derivative.

\emph{The continuous remainder.} $\zeta$ is continuous on $\D$ and
closed off the finite $Z(B)$, so $d\zeta = 0$ as a current: excising
disks $D_\varepsilon(p)$ and applying Stokes on the complement,
$\int_{\D \setminus \cup D_\varepsilon} d\psi \wedge \zeta$ reduces to
boundary terms $\oint_{\partial D_\varepsilon} \psi\,\zeta$, which
vanish as $\varepsilon \to 0$ since $\zeta$ is bounded near each $p$
and the circles have shrinking length. And
$d{\star}\zeta$ is, near $p$, the distributional Laplacian of the
$C^1$ potential $\arg g_p$ --- a genuine distributional Laplacian of a
single-valued function, with no excision subtlety, the smooth terms
$-\sum_{q \neq p} n_q\,d\varphi_q$ contributing nothing, their
potentials being harmonic near $p$ --- and away from
$Z(B)$ it is $\Delta\theta\,dx\,dy$. The distributional
Laplacian of a continuous $u$ passes the concentration test: with
$\psi_\varepsilon := \psi((\cdot - p)/\varepsilon)$,
\[
\langle \Delta u, \psi_\varepsilon \rangle
= \int u\,\Delta\psi_\varepsilon
= \int u(p + \varepsilon w)\,(\Delta\psi)(w)\; dw
\;\longrightarrow\; u(p)\int \Delta\psi = 0,
\]
by dominated convergence and $\int\Delta\psi = 0$; this is the claimed
absence of atoms in (ii), the model vortices contributing nothing to
$d{\star}\eta$ by the computation above. Adding the model
vortices to the remainder gives (i) and (ii); under the strengthened
hypotheses, $\Delta\arg g_p$ is continuous near $p$ ($g_p \in C^2$)
and $\Delta\theta$ is continuous off $Z(B)$
($B \in C^2(\D \setminus Z(B))$), the two matching where both are
defined, giving the continuous extension. The identification of the
total atomic mass with the charge is the localization
$n = \sum_p n_p$ of \cite[Prop.~4.1]{charge}, tame equations being
admissible since the finite $Z(B) \Subset \D$.
\end{proof}

\begin{corollary}[The charge is the flux of the phase form]\label{cor:flux}
For tame $B$ and any $\psi \in C^\infty_c(\D)$ with $\psi \equiv 1$ on
a neighborhood of $Z(B)$,
\[
n \;=\; \frac{1}{2\pi}\,\langle d\eta, \psi \rangle .
\]
This is the statement promised in Section~\ref{sec:massandcharge}:
the charge consumes the atoms of $d\eta$ and nothing else; the
curvature current $d{\star}\eta$ contributes nothing to it, and it
contributes nothing to the curvature.
\end{corollary}

\begin{theorem}[Completeness across the vortices]\label{thm:tame-complete}
Let $\Ecal_1, \Ecal_2$ be sourceless framed equations whose minimal forms
$B_1, B_2$ are phase-integrable per
Definition~\ref{def:phase-form}. Then $\Ecal_1$ and $\Ecal_2$ are equivalent
if and only if there is $F \in \Mob$ with
\[
\vartheta_1 = \vartheta_2 \circ F,
\qquad
d\eta_1 = F^*(d\eta_2),
\qquad
d{\star}\eta_1 = F^*(d{\star}\eta_2).
\]
On the tame sector the middle condition reads: equal charge measures,
$C_1 = F^* C_2$.
\end{theorem}

\begin{proof}
Necessity is Lemma~\ref{lem:covariance} together with the
equivariance of $\vartheta$ (Theorem~\ref{thm:equivariance}) and the
gauge--M\"obius normal form of residual morphisms
(proof of Theorem~\ref{thm:completeness}).

Sufficiency. Pull $B_2$ back: $\tilde B_2 := \overline{F'}(B_2 \circ F)$
is phase-integrable with
$(\vartheta, d\eta, d{\star}\eta)$-data equal to that of $B_1$
(Lemma~\ref{lem:covariance}), and $|\tilde B_2| = |B_1|$ follows from
the $\vartheta$-match as in Theorem~\ref{thm:completeness}; in
particular $Z(\tilde B_2) = Z(B_1) =: Z$, closed, null, with connected
complement.

Set $\alpha := \eta_1 - \tilde\eta_2 \in L^1_{\mathrm{loc}}(\D)$,
with $d\alpha = 0$ and $d{\star}\alpha = 0$ as currents. Writing
$\alpha = P\,dx + Q\,dy$, these say, distributionally,
$Q_x - P_y = 0$ and $P_x + Q_y = 0$: the Cauchy--Riemann system for
$f := P - iQ$, i.e.\ $\bar\partial f = 0$ in
$\mathcal{D}'(\D)$. By Weyl's lemma for $\bar\partial$, $f$ is
holomorphic; hence $\alpha$ is a smooth 1-form, closed and co-closed
--- the elliptic regularity of harmonic currents in the plane
\cite{derham} --- and on the simply connected $\D$ it is exact with
harmonic potential:
$\alpha = dh$, $\Delta h = 0$.

Now compare the fields directly through the single-valued unimodular
quotient
\[
E \;:=\; \frac{B_1\,\overline{\tilde B_2}}{|B_1|\,|\tilde B_2|}
\qquad \text{on } \D \setminus Z,
\]
which satisfies
$\operatorname{Im}(dE/E) = \eta_1 - \tilde\eta_2 = dh$ and
$|E| \equiv 1$, so $d\bigl(E\,e^{-ih}\bigr) = 0$: on the connected
$\D \setminus Z$, $E = e^{i(h + c)}$ for a single real constant $c$.
Hence $B_1 = e^{i(h + c)}\,\tilde B_2$ off $Z$, and on $Z$ both sides
vanish, so the identity holds on $\D$. As in
Theorem~\ref{thm:completeness}, $h + c$ is harmonic on the simply
connected $\D$, hence the phase $-2\arg\varphi$ of a holomorphic
zero-free gauge, and $B_1$ is gauge-equivalent to $\tilde B_2$, hence
equivalent to $B_2$ and $\Ecal_2$.
\end{proof}

\begin{example}[The charge measure is irreducible]\label{ex:zbar}
The triple cannot be thinned. On $\D$ take
\[
B_1 = z, \qquad B_2 = \bar z :
\]
both tame with the single vortex $0$, of local charges $+1$ and $-1$.
The moduli coincide, $|B_1| = |B_2| = |z|$, so
$\vartheta_1 = \vartheta_2$ --- every metric invariant agrees, and the
vanishing order at the vortex, which $\vartheta$ determines, is $1$
for both. The curvatures coincide as well: $\arg z$ and
$-\arg z$ are harmonic off the origin, so
$d{\star}\eta_1 = d{\star}\eta_2 = 0$. Only the charge measures
differ, $C_1 = \delta_0$ and $C_2 = -\delta_0$, and the equations are
indeed inequivalent, their charges $+1 \neq -1$ being invariants
\cite{charge}. The density $\vartheta$ sees the positions of the
vortices and their orders $|n_p|$; the \emph{signs} live in $d\eta$
alone. This is the precise sense in which the charge, absent from the
vortex-free pair, enters the complete data irreducibly.
\end{example}

\begin{remark}[The frontier]\label{rem:frontier}
Phase-integrability is exactly the hypothesis under which the
order-zero $L^1$ phase-current classification developed here is
defined. Within it, Theorem~\ref{thm:tame-complete} holds --- an orbit
separation: unlike the vortex-free sector, where
Proposition~\ref{prop:realization} also computes the range, the image
of the triple is not characterized here, the compatibility conditions
coupling $\vartheta$ to the currents across a vortex set
(Example~\ref{ex:zbar}) being left open --- and the
sector properly contains the tame one: $B = |z|^2$ has an isolated
zero admitting no factorization $\tau_p\,g_p$, yet $\eta \equiv 0$ is
integrable and the theorem applies. Beyond it, the order-zero
currents of Definition~\ref{def:phase-form} cease to exist: the
oscillation example of \cite[Rem.~4.8]{charge},
$B = |z|\,e^{i/|z|}$ near a vortex of local charge zero, has
$|\eta| = |d(1/\rho)| = \rho^{-2} \notin L^1_{\mathrm{loc}}$, so the
phase form defines neither pairing of
Definition~\ref{def:phase-form}. (In this charge-zero example the
single-valued phase $\theta = 1/\rho$ is itself locally integrable,
so derivatives of the \emph{distribution} $\theta$ survive at higher
order; but they are no longer represented by locally integrable
data, and the completeness argument, which consumes the phase at
order zero through Weyl's lemma, does not reach them.) There the germ
of the phase at the vortex is genuinely richer than the order-zero
data, and the classification of non-integrable vortex germs --- the
local moduli at the frontier --- is open. What the present section
settles is that the pathology begins exactly where the integral fails:
on the integrable side, three data modulo $\Mob$ suffice.
\end{remark}

\begin{remark}[Flat connection, Aharonov--Bohm fluxes]\label{rem:ab}
The physical reading sharpens at the vortices. The 1-form
$A := \tfrac12\,\eta$ is a connection which is \emph{flat} off the
vortex set --- $\eta$ is closed there --- and whose curvature
$dA = \tfrac12\,d\eta$ is, for tame $B$ and by
Theorem~\ref{thm:atoms}(i), the purely atomic 2-form
$dA = \pi \sum_p n_p\,\delta_p$: quantized Aharonov--Bohm fluxes
threading the vortices, with no distributed magnetic field. On the
general phase-integrable sector the curvature is a current, its
atomicity proved here for tame fields only. The phase
curvature $\kappa$ is \emph{not} the curvature of this connection ---
on the tame sector that curvature is entirely atomic --- but its
co-exterior counterpart, an invariant the flux picture cannot see. Section~\ref{sec:discussion}
takes up this reading.
\end{remark}

\section{The solution spaces}\label{sec:solutions}

The minimal form is the level at which the classical theory speaks of
solutions, and the moduli are represented at exactly this level; it is
natural to ask what the classification says about the solutions of
equivalent, and of inequivalent, minimal forms. The answer is
essentially complete, its one global caveat tracked precisely: a nontrivial
solution, where one exists, determines its equation; the sheaf of
local solutions --- nontrivial on every disk --- determines it
unconditionally; the moduli pair is computable from any one nonzero
solution; and the completeness theorem re-reads as a classification
of the solution sheaves, and of the global solution spaces on the
sector where these are nontrivial, under the geometric operators ---
multiplier times composition --- that the residual groupoid provides. A classical tension is resolved
along the way: the similarity principle makes every individual
solution resemble a holomorphic function, and the moduli measure
precisely the obstruction to doing for the whole space what similarity
does one solution at a time. And \emph{a pseudo-analytic
function theory} is meant here in Bers' sense \cite{bers}: the function
theory that one equation --- classically, one generating pair ---
determines, namely its solution space with the geometric morphisms;
the plural counts equations, not frameworks.

Throughout, $\Sol(B)$ denotes the space of
$w \in L^1_{\mathrm{loc}}(\D)$ solving the minimal form
$w_{\bar z} = B\bar w$ in $\mathcal{D}'(\D)$ --- the product
meaningful, $B$ being continuous. Every such $w$ lies in
$C^{1,\alpha}_{\mathrm{loc}}(\D)$, by an elliptic bootstrap through
the Pompeiu operator rather than by any direct Schauder estimate from
$L^1$: on a disk $D \Subset \D$, $w - T_D(B\bar w)$ is holomorphic
(Weyl's lemma for $\bar\partial$), and the mapping properties of
$T_D$ \cite[Ch.~I, \S\S 5--9]{vekua} --- $L^1$ into $L^q$ for every $q < 2$, by
Young's inequality against the kernel $1/\pi z \in L^q_{\mathrm{loc}}$;
$L^p$ into $L^r$ for every $\tfrac1r > \tfrac1p - \tfrac12$ while
$p < 2$; into $C^{0,1 - 2/p}$ once $p > 2$; and $C^{\beta}$ into
$C^{1,\beta}$ --- climb in finitely many steps from
$w \in L^1_{\mathrm{loc}}$ through $w$ continuous to
$w \in C^{1,\alpha}_{\mathrm{loc}}$, the H\"older regularity of $B$
entering at the last step. It is naturally a vector space over $\R$
and is generally not a complex vector space; the analytic sector is
the exceptional case, as Remark~\ref{rem:poles} makes precise.

\begin{proposition}[One solution determines the equation]\label{prop:one-solution}
Let $B$ be a minimal form and $w \in \Sol(B)$, $w \not\equiv 0$. Then
$Z(w)$ is discrete in $\D$, and
\begin{equation}\label{eq:recover}
B \;=\; \frac{w_{\bar z}}{\bar w} \qquad \text{on } \D \setminus Z(w),
\end{equation}
so that $B$ is determined on all of $\D$ by continuity. Consequently:
\begin{itemize}
\item[(i)] distinct minimal forms share no solution but $0$;
consequently the assignment $B \mapsto \Sol(B)$ is injective on every
subclass on which $\Sol \neq \{0\}$ is known
(Remark~\ref{rem:nontrivial}), and unconditionally at the level of
local solutions (Proposition~\ref{prop:sheaf});
\item[(ii)] the moduli are computable from any one nontrivial
solution: for vortex-free $B$ and any $w \in \Sol(B)$,
$w \not\equiv 0$,
\[
\vartheta \;=\; \tfrac14\,(1-|z|^2)^2\,\frac{|w_{\bar z}|^2}{|w|^2}
\qquad \text{on } \D \setminus Z(w),
\]
and, after extending the quotient in \eqref{eq:recover} continuously
across the discrete set $Z(w)$,
\[
K \;=\; \Delta\arg\!\left(\frac{w_{\bar z}}{\bar w}\right)\,dx\,dy
\qquad \text{in } \mathcal D'(\D).
\]
When the extended phase admits a $C^2$ branch, this becomes
$\kappa = \Delta_{\mathrm{hyp}}\arg(w_{\bar z}/\bar w)$. The data are
independent of the choice of $w$, since by \eqref{eq:recover} they are
functions of $B$ alone.
\end{itemize}
\end{proposition}

\begin{proof}
On any disk $D \Subset \D$ the coefficient is bounded and H\"older,
and the similarity principle applies: $w = \chi e^{\omega}$ on $D$
with $\chi$ holomorphic and $\omega$ continuous
\cite[Ch.~III]{vekua}, \cite{bers}.
The set of points of $\D$ near which $w$ vanishes identically is open,
and closed in $\D$: at any of its boundary points a similarity disk
gives a holomorphic $\chi$ vanishing on an open subset, hence
identically, hence $w \equiv 0$ near that point. By connectedness
$w \equiv 0$ --- excluded --- or no such point exists, and then every
zero is a zero of some $\chi$, hence isolated. Off $Z(w)$ the equation
divides by $\bar w \neq 0$, giving \eqref{eq:recover}; $B$ is
continuous on $\D$ and $\D \setminus Z(w)$ is dense. (i) A common
nontrivial solution computes both coefficients by
\eqref{eq:recover}; and if $\Sol(B_1) = \Sol(B_2) \neq \{0\}$, any
nonzero member is common. (ii) Substitution of \eqref{eq:recover}
into Definition~\ref{def:pair}.
\end{proof}

The injectivity of (i) is conditional at the global level only. The
local theory removes the condition: nontrivial solutions always exist
locally, and the local solutions already determine the coefficient.

\begin{proposition}[The local solution sheaf determines the equation]\label{prop:sheaf}
For open $U \subseteq \D$ let $\Sol_U(B)$ denote the solutions of
$w_{\bar z} = B\bar w$ in $\mathcal{D}'(U)$, so that
$\Sol(B) = \Sol_\D(B)$ and restriction maps
$\Sol_U(B) \to \Sol_V(B)$ for $V \subseteq U$ make
$U \mapsto \Sol_U(B)$ a sheaf of real vector spaces. Then:
\begin{itemize}
\item[(i)] every disk $D \Subset \D$ has
$\Sol_D(B) \neq \{0\}$ --- indeed infinite-dimensional
\cite{bers, vekua} (Remark~\ref{rem:nontrivial});
\item[(ii)] if $\Sol_D(B_1) = \Sol_D(B_2)$ for every disk $D$ in some
cover of $\D$ by disks $D \Subset \D$, then $B_1 = B_2$.
\end{itemize}
Consequently the assignment $B \mapsto \bigl(U \mapsto \Sol_U(B)\bigr)$,
from minimal forms to solution sheaves, is injective --- with no
global nontriviality hypothesis.
\end{proposition}

\begin{proof}
(i) is the local theory: on $D \Subset \D$ the coefficient is bounded
and H\"older, and the solution space is infinite-dimensional
\cite{bers, vekua}. (ii) On each disk $D$ of the cover pick
$w \in \Sol_D(B_1) = \Sol_D(B_2)$, $w \not\equiv 0$; the argument of
Proposition~\ref{prop:one-solution}, run on $D$, gives $Z(w)$
discrete in $D$ and $B_1 = w_{\bar z}/\bar w = B_2$ on
$D \setminus Z(w)$, hence on $D$ by continuity. The disks cover $\D$.
\end{proof}

\begin{theorem}[The solution spaces classify]\label{thm:solutions}
Let $B_1, B_2$ be minimal forms with $\Sol(B_2) \neq \{0\}$. The
following are equivalent:
\begin{itemize}
\item[(a)] $B_1$ and $B_2$ are equivalent, i.e.\ related by
\eqref{eq:orbit};
\item[(b)] $\Sol(B_1) = \varphi\cdot\bigl(\Sol(B_2) \circ F\bigr)$ for
some holomorphic zero-free $\varphi$ on $\D$ and some $F \in \Mob$;
\item[(c)] $\varphi\cdot(w \circ F) \in \Sol(B_1)$ for a
\emph{single} $w \in \Sol(B_2)$, $w \not\equiv 0$, and some such pair
$(\varphi, F)$.
\end{itemize}
\end{theorem}

\begin{proof}
(a) $\Rightarrow$ (b). Let
$B_1 = (\bar\varphi/\varphi)\,\overline{F'}\,(B_2 \circ F)$ per
\eqref{eq:orbit}. For $w \in \Sol(B_2)$ set
$w_1 := \varphi^{-1}\,(w \circ F)$; since $\varphi^{-1}$ is
holomorphic,
\[
(w_1)_{\bar z}
= \varphi^{-1}\,\overline{F'}\,(w_{\bar z} \circ F)
= \varphi^{-1}\,\overline{F'}\,(B_2 \circ F)\,(\bar w \circ F)
= \frac{\bar\varphi}{\varphi}\,\overline{F'}\,(B_2 \circ F)\cdot
\overline{\varphi^{-1}\,(w \circ F)}
= B_1\,\bar w_1 ,
\]
so $w \mapsto \varphi^{-1}(w \circ F)$ maps $\Sol(B_2)$ into
$\Sol(B_1)$, bijectively, with inverse
$w_1 \mapsto (\varphi \circ F^{-1})\,(w_1 \circ F^{-1})$ of the same
shape --- the multiplier transported along with the composition, so
that the composite returns
$\varphi^{-1}\,(\varphi \circ F^{-1} \circ F)\,w = w$; this is
(b) with multiplier $\varphi^{-1}$.

(b) $\Rightarrow$ (c). Take any nontrivial $w \in \Sol(B_2)$; its
image is nontrivial, the multiplier being zero-free and $F$ bijective.

(c) $\Rightarrow$ (a). The transported solution
$w_1 := \varphi\,(w \circ F)$ is nontrivial, and off the discrete
$Z(w_1)$, Proposition~\ref{prop:one-solution} recovers
\[
B_1
= \frac{(w_1)_{\bar z}}{\bar w_1}
= \frac{\varphi\,\overline{F'}\,(B_2 \circ F)\,(\bar w \circ F)}
       {\bar\varphi\,(\bar w \circ F)}
= \frac{\varphi}{\bar\varphi}\;\overline{F'}\,(B_2 \circ F),
\]
hence on all of $\D$ by continuity: the relation \eqref{eq:orbit} with
gauge $\varphi^{-1}$.
\end{proof}

The hypothesis $\Sol(B_2) \neq \{0\}$ enters part (c) alone, and
localization removes it: since nontrivial solutions always exist on
disks, the sheaf-level statement is unconditional.

\begin{corollary}[The solution sheaves classify, unconditionally]\label{cor:sheaf-classify}
For minimal forms $B_1, B_2$ the following are equivalent, with no
nontriviality hypothesis:
\begin{itemize}
\item[(a)] $B_1$ and $B_2$ are equivalent, i.e.\ related by
\eqref{eq:orbit};
\item[(b)] there are a holomorphic zero-free $\varphi$ on $\D$ and
$F \in \Mob$ with
\[
\Sol_U(B_1) \;=\; \varphi\cdot\bigl(\Sol_{F(U)}(B_2) \circ F\bigr)
\qquad \text{for every open } U \subseteq \D ;
\]
\item[(c)] the displayed identity holds, for a single such pair
$(\varphi, F)$, on every disk of some cover of $\D$ by disks
$D \Subset \D$.
\end{itemize}
\end{corollary}

\begin{proof}
(a) $\Rightarrow$ (b): the computation of
Theorem~\ref{thm:solutions}, (a) $\Rightarrow$ (b), is local and runs
verbatim over each $U$, the transport
$w \mapsto \varphi^{-1}(w \circ F)$ carrying $\Sol_{F(U)}(B_2)$
bijectively onto $\Sol_U(B_1)$. (b) $\Rightarrow$ (c) is a
restriction. (c) $\Rightarrow$ (a): on each disk $D$ of the cover
pick $w \in \Sol_{F(D)}(B_2)$ nontrivial
(Proposition~\ref{prop:sheaf}(i), $F(D) \Subset \D$ since $F$ is an
automorphism); its transport is a nontrivial member of $\Sol_D(B_1)$,
and the computation of Theorem~\ref{thm:solutions},
(c) $\Rightarrow$ (a), run on $D$, recovers
$B_1 = (\varphi/\bar\varphi)\,\overline{F'}\,(B_2 \circ F)$ on $D$.
The pair $(\varphi, F)$ being one and the same on every disk, the
identities patch over the cover, and \eqref{eq:orbit} holds on $\D$
with gauge $\varphi^{-1}$.
\end{proof}

A single disk does not suffice in (c): the identity on one disk
determines the orbit relation there and nowhere else, the coefficient
being merely continuous; it is the cover, against one global
$(\varphi, F)$, that globalizes.

Combined with Theorem~\ref{thm:completeness}, on the vortex-free
sector: two minimal forms carry solution sheaves related by a
multiplier--composition isomorphism if and only if their pairs
$(\vartheta, K)$ agree modulo a common M\"obius map
(Corollary~\ref{cor:sheaf-classify}); at the level of the global
spaces the same holds whenever $\Sol(B_2) \neq \{0\}$
(Theorem~\ref{thm:solutions}). Inequivalent
minimal forms are genuinely different pseudo-analytic function
theories on $\D$ --- sharing no function but $0$, by
Proposition~\ref{prop:one-solution}(i) --- and the pair is the
complete invariant of the function theory. Abstract linear isomorphism
of the spaces, by contrast, is not the equivalence considered here:
it discards the geometric action of the equation, and no
classification of the underlying real vector spaces is attempted; the
geometric
operators, the maps the groupoid itself provides, are the correct
morphisms of function theories, and part (c) says one function
suffices to test them.

The massless point deserves its own statement. The sector
$\Mcal = 0$ is the single orbit $B \equiv 0$ --- the analytic
equation \cite{mass} --- and it sits outside the hypotheses of both
completeness theorems: $Z(B) = \D$ is neither empty nor null, so the
equation is neither vortex-free nor phase-integrable. It needs
neither, being cut out by the mass alone; and the specialization of
Theorem~\ref{thm:solutions}(c) at $B_2 = 0$ collapses to a one-line
criterion.

\begin{corollary}[No holomorphic solutions off the analytic sector]\label{cor:massless}
A minimal form admitting a single holomorphic solution
$w \not\equiv 0$ is the analytic equation: $B \equiv 0$, and
$\Mcal = 0$. Equivalently, while every solution of every minimal
form is \emph{nearly} holomorphic --- $w = \chi e^{\omega}$ with
$\omega$ continuous, by similarity --- an \emph{exactly} holomorphic
solution occurs only at mass zero: the similarity exponent is not
removable from even one solution while $\Mcal > 0$.
\end{corollary}

\begin{proof}
For holomorphic $w \not\equiv 0$, equation \eqref{eq:recover} gives
$B = w_{\bar z}/\bar w = 0$ off the discrete $Z(w)$, hence on $\D$ by
continuity; $\Mcal = 0$ is then Proposition~\ref{prop:mass-moment}.
\end{proof}

\begin{remark}[Similarity is solutionwise; the two poles]\label{rem:poles}
The classical instinct against the theorem --- all these solution
theories look alike, since by similarity every solution resembles a
holomorphic function --- fails at the quantifier. The similarity
multiplier $e^{\omega}$ is built \emph{from the solution it
trivializes}: $\omega$ is a Pompeiu integral of $B\,\bar w/w$
\cite{vekua, bers}, a different multiplier for each $w$, and it maps
one solution, not the space. The multipliers that map the space are
exactly the holomorphic gauges of the residual groupoid, and the
moduli measure the obstruction to doing simultaneously what
similarity does one solution at a time. The two poles are classical
statements: full simultaneous trivialization ---
$\Sol(B) = \varphi\cdot(\mathcal{O}(\D) \circ F)$, the holomorphic
theory --- holds if and only if $B \equiv 0$, the analytic sector of
zero mass \cite{mass}; simultaneous normalization to a positive
coefficient holds if and only if $\kappa \equiv 0$
(Corollary~\ref{cor:positive}). Between the poles the pair
interpolates, one function theory per point of the moduli. The
$\R$-linear structure repeats the first pole in miniature: $iw$ solves
whenever $w$ does if and only if $B\,\bar w \equiv 0$ for every
$w \in \Sol(B)$, and a nontrivial solution has discrete zeros
(Proposition~\ref{prop:one-solution}); hence, whenever
$\Sol(B) \neq \{0\}$ --- as for bounded $B$,
Proposition~\ref{prop:zero-free} --- $\Sol(B)$ is a complex vector space
precisely when $B \equiv 0$: the holomorphic theory is the unique
nontrivial $\C$-linear one. It is also the unique fixed
point of the entire residual groupoid,
$(\bar\varphi/\varphi)\,\overline{F'}\cdot 0 = 0$ for every gauge and
every M\"obius map: maximal symmetry of the equation matching the
maximal family of geometric self-maps of its solution space, the
correspondence of Theorem~\ref{thm:solutions} working at its extreme
point.
\end{remark}

\begin{remark}[Upstream and downstream of the equation]\label{rem:upstream}
The comparison with Bers sharpens here. Bers calls two generating
pairs \emph{equipotent} when they determine the same class of
pseudo-analytic functions \cite{bers}: since, by
Proposition~\ref{prop:one-solution}, the solution space determines the
equation, equipotence is exactly the fiber of the passage from pairs
to equations. The classical theory quotients the pairs down to the
equation; the present paper quotients the equations down to the
moduli. The two theories stand in series --- pairs, equations, moduli
--- Bers' equivalences filling the fibers of the first arrow, the
residual groupoid the fibers of the second.
\end{remark}

\begin{proposition}[Global zero-free solutions for bounded fields]\label{prop:zero-free}
If $B$ is bounded on $\D$, the minimal form $w_{\bar z} = B\bar w$
admits a global zero-free solution $w = e^{\sigma}$ with
$\sigma \in C^0(\overline\D)$; in particular $\Sol(B) \neq \{0\}$.
\end{proposition}

\begin{proof}
We use a fixed-point argument rather than the similarity principle,
since the latter represents a solution already in hand and is not by
itself an existence theorem. Seek $w = e^{\sigma}$: the equation
becomes $\bar\partial\sigma = B\,e^{\bar\sigma - \sigma}$. Define
\[
\mathcal T(\sigma)
\;:=\; T_\D\bigl(B\,e^{\bar\sigma - \sigma}\bigr)
\qquad\text{on } C^0(\overline\D).
\]
The integrand has modulus $|B| \leq \|B\|_\infty$ pointwise,
independently of $\sigma$, and $T_\D$ carries bounded sets of
$L^\infty$ into bounded sets of $C^{0,\beta}(\overline\D)$
\cite[Ch.~I, \S\S 5--9]{vekua}. Hence there is $R>0$ such that
$\|\mathcal T(\sigma)\|_{C^0} \leq R$ for every $\sigma$, and the
image of $\mathcal T$ is relatively compact in
$C^0(\overline\D)$. Thus $\mathcal T$ is continuous and compact and
maps the closed convex ball
$\{\sigma : \|\sigma\|_{C^0} \leq R\}$ into itself. Schauder's fixed
point theorem produces $\sigma$ with
$\bar\partial\sigma = B e^{\bar\sigma - \sigma}$ in
$\mathcal{D}'(\D)$. The fixed point lies in
$W^{1,p}_{\mathrm{loc}}(\D)$ for every $p < \infty$ --- the
derivative of the Pompeiu integral of a bounded function is its
Calder\'on--Zygmund transform --- so the chain rule
$\bar\partial e^{\sigma} = e^{\sigma}\,\bar\partial\sigma$ is valid,
and $w = e^{\sigma}$ solves the minimal form, zero-free.
\end{proof}

\begin{remark}[Nontriviality]\label{rem:nontrivial}
The hypothesis $\Sol(B_2) \neq \{0\}$ of Theorem~\ref{thm:solutions}
is stated where it is used: part (c) presumes a solution in hand. It
holds on every relatively compact subdomain, where the solution space
is infinite-dimensional \cite{bers, vekua}, and globally on $\D$
whenever $B$ is bounded (Proposition~\ref{prop:zero-free}). Whether
the bare class of Definition~\ref{def:class} always admits a global
nontrivial solution is a question at the boundary, of the same species
as the trace problems of \cite[\S 7]{charge}, and we do not pursue it;
Proposition~\ref{prop:sheaf} and Corollary~\ref{cor:sheaf-classify}
circumvent it by localizing, and nothing in the classification waits
on its answer.
\end{remark}

\section{Discussion}\label{sec:discussion}

\emph{The physical reading.} At the minimal slice the moduli admit a
gauge-theoretic dictionary --- a translation of objects, not a claim
of physics: no action, no dynamics, only the transformation laws.
The density $\vartheta$ plays the role of an energy profile in
hyperbolic clothing, of which $\Mcal$ retains the total; the 1-form
$A := \tfrac12\eta$ transforms, under the residual gauges and by
Lemma~\ref{lem:covariance}, exactly as a $U(1)$ connection form
under gauge transformations with harmonic phase, and we read it as
one: flat off the vortices, with field strength on the tame sector
the purely atomic $dA = \pi\sum_p n_p\,\delta_p$
(Remark~\ref{rem:ab}) --- Aharonov--Bohm fluxes in units of $\pi$,
the coefficient field carrying gauge charge two, threading the
vortices with no distributed field: the charge. The question is
where $\kappa$ lives, since it is not the curvature of $A$: the
answer is that the gauge freedom the class grants the connection is
smaller than the dictionary's. Gauge theory identifies connections
differing by $d\chi$ for arbitrary $\chi$; the residual gauges act
on $A$ by exactly the shifts $d\chi$ with $\chi$ harmonic
(Proposition~\ref{prop:weight-gauge}), and the quotient therefore
remembers the co-derivative $d{\star}A = \tfrac12\kappa\,\hyparea$,
which an arbitrary $\chi$ would eliminate --- the Coulomb gauge ---
and a harmonic $\chi$ cannot touch. The phase curvature is a
Coulomb-gauge residue promoted to an invariant: by
Corollary~\ref{cor:positive} it is exactly the obstruction to
trivializing the connection by an admissible gauge, and
Example~\ref{ex:phase} is two equations agreeing in everything the
dictionary expresses --- identical profile $\vartheta$, identical
(vanishing) field strength $dA$ --- distinguished by the
co-curvature alone, an invariant for which the dictionary has no
word. The moduli of the equation are strictly finer than its
gauge-theoretic translation.

\emph{Why there is no Bogomolny bound, again.} The absence of an
energy--vortex inequality was traced in \cite[Rem.~6.3]{charge} to the
mass containing no derivative of the field; the moduli make the
diagnosis structural. The Ginzburg--Landau energy reads the phase form
--- the kinetic term $|\nabla u|^2$ contains $|\eta|^2$ --- and the
Bogomolny rearrangement converts exactly that term into flux
\cite[Rem.~6.3]{charge}. The pseudo-analytic mass reads $\vartheta$
alone, and the decisive freedom is a scaling: for $\lambda > 0$ the
minimal field $\lambda B$ has mass $\lambda^2 \Mcal$, while its phase
form --- $\operatorname{Im}\bigl(d(\lambda B)/(\lambda B)\bigr) = \eta$,
blind to $\lambda$ --- and with it the vortex set and every local
charge, is untouched. At fixed nonzero charge the mass is therefore
arbitrarily small, and arbitrarily large, and no bound
$\Mcal \geq c(n) > 0$ can hold. Note that this needs no independence
of the fields, which indeed fails across the vortices ---
$\vartheta$ reads the vortex positions and the absolute orders
$|n_p|$ (Example~\ref{ex:zbar}), and only the vortex-free pair is
fully unconstrained (Proposition~\ref{prop:realization}); one ray in
the moduli already forbids the inequality.

\emph{Multiply connected domains.} Three ingredients change, each
visibly. The uniformization lands on a circular domain rather than
$\D$, and the residual conformal group collapses: for an annulus,
rotations together with the boundary-swapping inversion --- the
everting map of \cite[Rem.~3.6]{charge} --- and for connectivity three
or more, a finite group. The gauge phases acquire periods: a zero-free
holomorphic $\varphi$ on a domain with $m$ holes has
$\oint_{\gamma_j} d(\arg\varphi) \in 2\pi\Z$, so the admissible phase
shifts $-2\arg\varphi$ are harmonic functions with periods in the even
lattice $4\pi\Z$, and the periods of the phase form ---
$\oint_{\gamma_j}\eta = 2\pi n_j$, the component charges --- are
invariants exactly in $\Z/2\Z$, with the total charge exact since a
zero-free $\varphi$ has total boundary winding zero. This is the
exact/mod-$2$ dichotomy of \cite[Rem.~3.6]{charge}, with the lattice
that causes it made visible. And the completeness argument loses its
last step: a closed and co-closed 1-form is no longer exact, the
harmonic 1-forms contributing an $m$-dimensional space, so the triple
must be supplemented by the periods of $\eta$ modulo the gauge
lattice. The expected statement --- complete data
$(\vartheta,\, d\eta,\, d{\star}\eta,\, \text{periods mod } 4\pi\Z)$
modulo the finite conformal group --- we leave to future work.

\emph{Measurable regularity.} The descent to Bojarski's class
\cite{bojarski}, where
\cite{framed} lives, attacks the data and not the symmetries. The
residual groupoid is unchanged: a $1$-quasiconformal homeomorphism is
conformal and a $W^{1,2}_{\mathrm{loc}}$ gauge with
$\varphi_{\bar z} = 0$ almost everywhere is holomorphic, by the two
Weyl lemmas \cite{aim}, so measurable minimal forms are still
classified modulo holomorphic gauges and $\Mob$. The modulus half then
descends intact: with lower-order data in $L^p_{\mathrm{loc}}$,
$p > 2$ --- the classical threshold below which the similarity gauge
$\varphi = e^s$ need not stay bounded --- the minimal reduction runs
at measurable $\mu$, and $\vartheta$, its spectrum
(Proposition~\ref{prop:spectrum}), and the mass identity of
Proposition~\ref{prop:mass-moment} survive almost everywhere. The
phase half is the open half. At bare $L^2_{\mathrm{loc}}$ the field
has no derivative, the phase form does not exist, and the situation is
that of \cite[\S 7]{charge}, whose conjecture on the quasiconformal
invariance of the winding is the group-level shadow of the missing
phase theory. Phase-integrability is plainly the natural intermediate
class --- Theorem~\ref{thm:tame-complete} uses smoothness nowhere
essential, and its Weyl step already runs at $L^1_{\mathrm{loc}}$ ---
but the comparison quotient of its proof divides by $|B|$, and the
sharp statement for $W^{1,1}_{\mathrm{loc}}$ fields with unbounded
zero-set geometry we leave open.

\section{Locality} The admissible morphisms split by their definability. The
recombinations of the unknown and the scalings are \emph{pointwise}:
the group element at a point acts through finitely many numbers at
that point, and a quotient by a pointwise group is algebra done
fiberwise --- which is why the absorption theorem of \cite{absorption} and
the minimal-form algebra of Section~\ref{sec:reduction} are formulas. The
residual symmetries are not pointwise: a holomorphic gauge and a
conformal chart are defined by a differential equation --- they are
kernels of $\bar\partial$ --- and a quotient by a pseudogroup cut out
by $\bar\partial$ is taken by inverting $\bar\partial$. The reduction
pays accordingly, in exactly the two places the H\"older exponent is
spent: the $\Acal$-gauge is a $\bar\partial$-primitive, the
uniformizing chart a Beltrami solve composed with a Riemann map. The invariants are precisely as
non-local as the automorphisms that define them.

The non-locality, moreover, concentrates in a single object, and the
concentration is a provable statement rather than an impression. The
mass $2$-form $\Theta$ is a pointwise expression in the coefficients
and one derivative of the frame \cite{framed}; the conformal
structure is carried pointwise by $\mu$; and the phase currents,
though defined through the minimal field --- the output of the
$\bar\partial$-gauge of Step 3 --- do not in fact require that solve:

\begin{proposition}[The phase currents are computable before the gauge]\label{prop:gauge-free}
Let $w_{\bar z} + \Acal w + \Bcal\bar w = 0$ on $\D$ be a slice
equation as in Step 2 of Theorem~\ref{thm:reduction},
$\Acal, \Bcal \in C^{\alpha}_{\mathrm{loc}}(\D)$, and let
$B = -\Bcal\,e^{\bar s - s}$ be the minimal field produced by a
primitive $\bar\partial s = -\Acal$ (Step 3, Remark~\ref{rem:sign}).
Then:
\begin{itemize}
\item[(i)] $|B| = |\Bcal|$: the modulus data --- $\vartheta$, its
moments, the mass form $\Theta$ --- read $\Bcal$ alone;
\item[(ii)] if $\Bcal \in C^1(\D \setminus Z(\Bcal))$ --- the
regularity supplied as in Definition~\ref{def:phase-form} --- then
$Z(B) = Z(\Bcal)$, and $\eta_B \in L^1_{\mathrm{loc}}(\D)$ if and
only if $\eta_{\Bcal} \in L^1_{\mathrm{loc}}(\D)$; when the common
zero set moreover satisfies the standing hypotheses of
Definition~\ref{def:phase-form} --- closed, Lebesgue-null, connected
complement --- this is phase-integrability of either field. The
phase forms differ by a continuous exact form,
\[
\eta_B \;=\; \eta_{\Bcal} \;-\; 2\,d\operatorname{Im}s,
\qquad
\eta_{\Bcal} := \operatorname{Im}\frac{d\Bcal}{\Bcal}
\]
(the sign of Remark~\ref{rem:sign} immaterial,
$d(-\Bcal)/(-\Bcal) = d\Bcal/\Bcal$), and the two currents obey
\[
d\eta_B \;=\; d\eta_{\Bcal},
\qquad
d{\star}\eta_B \;=\; d{\star}\eta_{\Bcal}
\;+\; 8\operatorname{Im}(\partial\Acal)\,dx\,dy,
\]
the last term read distributionally,
$\langle \operatorname{Im}\partial\Acal, \psi\rangle
:= -\int_\D \operatorname{Im}(\Acal\,\partial\psi)\,dx\,dy$.
\end{itemize}
In particular both currents --- the charge current and the curvature
current --- are independent of the choice of primitive $s$ and
computable from $(\Acal, \Bcal)$ by local formulas, with no
$\bar\partial$-solve performed.
\end{proposition}

\begin{proof}
(i) is $|e^{\bar s - s}| = 1$. For (ii):
$s \in C^{1,\alpha}_{\mathrm{loc}}$ (Step 3), so
$d\operatorname{Im}s$ is a continuous 1-form, and taking imaginary
parts in $dB/B = d\Bcal/\Bcal + d\bar s - ds$ gives the displayed
relation, $\operatorname{Im}(d\bar s - ds) = -2\,d\operatorname{Im}s$;
local integrability passes both ways across the continuous
difference. The first law is $d(d\operatorname{Im}s) = 0$, valid
distributionally for the $C^1$ function $\operatorname{Im}s$. For the
second, $d{\star}d\operatorname{Im}s = \Delta(\operatorname{Im}s)\,dx\,dy$
as currents, and
$\Delta s = 4\,\partial\bar\partial s = -4\,\partial\Acal$ in
$\mathcal{D}'(\D)$, so
$\Delta\operatorname{Im}s = -4\operatorname{Im}(\partial\Acal)$ and
$-2\,d{\star}d\operatorname{Im}s = 8\operatorname{Im}(\partial\Acal)\,dx\,dy$.
Independence of the primitive is read off the formulas, which do not
contain $s$; directly, two primitives differ by a holomorphic $g$,
whose contribution $-2\,d\operatorname{Im}g$ is a smooth exact form
with harmonic potential, killed by both $d$ and $d{\star}$ as in
Lemma~\ref{lem:covariance}.
\end{proof}

What is not local is the hyperbolic reference $\hyparea$: the Poincar\'e density strictly
decreases under enlargement of the domain, so its value at a point
sees the entire boundary, and no expression in finitely many
derivatives of the coefficients at the point reproduces it. The
moduli are ratios of local currents against this one non-local
reference,
\[
\vartheta \;=\; \frac{\Theta}{\hyparea},
\qquad
\kappa \;=\; \frac{d{\star}\eta}{\hyparea},
\]
the numerators formulas --- Proposition~\ref{prop:gauge-free} ---
the denominator a solve.

This division of labor selects the two invariants that began the
program. In the zeroth moment the reference cancels,
\[
\Mcal \;=\; \int_\D \vartheta\,\hyparea \;=\; \int_\Omega \Theta,
\]
a local functional, computable before any uniformization
(Proposition~\ref{prop:mass-moment}); and, whenever the numerical
charge is defined, it is topological --- a winding, blind to every
metric, likewise computable pre-uniformization \cite{charge}. Among
the numerical invariants considered in this paper, mass and charge are
the two that eliminate the hyperbolic reference: every
higher moment $I_\Phi$, every M\"obius-invariant functional of the
shape of $\vartheta$, and the whole of $\kappa$ retain the reference
and inherit its non-locality. The incompleteness of
Section~\ref{sec:incompleteness} thus acquires a second reading:
\emph{the locally computable invariants of a sourceless framed
equation are incomplete, and completeness costs exactly one non-local
uniformization problem: the construction of the conformal disk chart
--- here a Beltrami chart followed by a Riemann map --- which
produces the hyperbolic reference}; the $\bar\partial$-gauge, by
Proposition~\ref{prop:gauge-free}, need never be performed.

This suggests a Cartan-type local normalization --- scalar invariants
and their functional relations --- generically and for the
\emph{modulus half} of the data alone: on
$\{\Theta \neq 0\}$ the conformal class of $\mu$ contains a unique
metric with area form $\Theta$, and its curvature is a pointwise
invariant --- of $\vartheta$ and its derivatives, the phase escaping
the construction entirely. The construction moreover degenerates
exactly where the theory's content lives --- on the zero set of the
numerator field, which carries the charge, and at the massless
equation, which carries the entire residual symmetry --- while the
hyperbolic normalization is uniform across these strata. The
reference is the price of the uniformity; the precedent is classical:
curvature is local, the uniformizing map and Riemann's moduli are
not.

\emph{Open problems.} Five, in increasing distance from the plane.
\emph{(1)} The multiply connected moduli, as sketched above.
\emph{(2)} The local moduli at the frontier: classify the
non-phase-integrable vortex germs, of which the oscillation example is
the first specimen; Remark~\ref{rem:frontier} shows the current-level
data end exactly there. \emph{(3)} Disconnecting zero sets: when
$Z(B)$ is null but separates the domain, the sufficiency argument of
Theorem~\ref{thm:tame-complete} leaves one phase constant per
component, and equations agreeing in the whole triple yet differing by
relative phases across a separating zero set appear possible; whether
these constants are genuine additional moduli is open. \emph{(4)} The
boundary: the charge paper left open the interaction of $n$ with the
index of Vekua's Riemann--Hilbert problem \cite{vekua, charge}; the
question now refines to fields --- how do the solvability dimensions
of boundary value problems read $(\vartheta, \kappa, C)$, and in
particular the local data at a vortex of prescribed atom?
\emph{(5)} Out of the plane: in the quaternionic extension of the
program the numerator field becomes higher-dimensional and the winding
a mapping degree \cite{charge}; the present paper sharpens the
question, since the split of the phase derivative into exterior and
co-exterior currents used the conformal invariance of the Hodge star in
the middle degree --- available again, in dimension four, on 2-forms.
On the tame planar sector the former is atomic and the latter has no
atomic part; beyond it their singular structures are part of the
problem. Whether the two-current structure, and
with it a completeness theorem, survives the loss of commutativity is
the precise question.

\section{Coda: from pseudo-analytic equations to pseudo-analytic theories}\label{sec:coda}

The program of \cite{mass, framed, charge} reads the
coefficients of an equation as coordinates on a space of equations and
asks for the geometry of that space under its own symmetries. The
classification just computed can be stated naturally at the level of
the function theories themselves.

\begin{definition}[Pseudo-analytic theory]\label{def:theory}
A \emph{pseudo-analytic theory} is a minimal pair $(\D, B)$ --- the
unit disk with a minimal form per Theorem~\ref{thm:reduction} ---
identified with the function theory it determines in Bers' sense: the
sheaf of local solutions, with the geometric morphisms of
Section~\ref{sec:solutions}. Every sourceless framed equation
$(\Omega, \Ecal)$ per Definition~\ref{def:class} determines a theory,
its minimal reduction, uniquely up to the residual groupoid
(Remark~\ref{rem:well-defined}). Two theories are \emph{isomorphic}
if their minimal forms are equivalent --- by
Corollary~\ref{cor:sheaf-classify}, equivalently, if their solution
sheaves admit a multiplier--composition isomorphism; on the sector of
global nontriviality, equivalently again, if their global solution
spaces do (Theorem~\ref{thm:solutions}).
\end{definition}

The restriction to the minimal slice in this definition is not a
loss: the solution spaces of a general framed equation are related to
those of its minimal form by the real-linear recombinations
$w = \varphi w' + \psi \bar w'$ of the reduction, and it is on the
minimal slice, where the stabilizer forces $\psi = 0$
(Proposition~\ref{prop:residual}), that the morphisms take the
multiplier--composition shape the classification uses.

In this noun the paper reads as four statements about theories.
\emph{One:} every sourceless framed equation on a bounded simply
connected domain determines a theory $(\D, B)$ --- the reduction
consumes the domain, Riemann's theorem acting as the rigidity of the
first component over simply connected bases, and one field remains,
unique up to the residual groupoid
(Theorem~\ref{thm:reduction}, Remark~\ref{rem:well-defined}). \emph{Two:} mass and charge are
functions on the space of theories, not parameters within any one
theory; a family such as the constant-coefficient equations is a
curve in the space, transverse to the level sets of the mass.
\emph{Three:} the theories are classified, on the phase-integrable
sector of Definition~\ref{def:phase-form}, by the triple
$(\vartheta,\, d\eta,\, d{\star}\eta)$ on the hyperbolic
disk modulo its isometries
(Theorems~\ref{thm:completeness}, \ref{thm:tame-complete}); on the
vortex-free sector the classifying pair is moreover unconstrained and
every H\"older pair with positive first component occurs
(Proposition~\ref{prop:realization}), while across
the vortices the triple is constrained --- $\vartheta$ already reads
the positions and orders of the vortices, and the charge current must
be compatible with them (Example~\ref{ex:zbar}) --- and the two
invariants that began the program, the mass of \cite{mass, framed}
and the charge of \cite{charge}, are its two numerical projections:
the total of the first field and, whenever the numerical charge is
defined, the normalized total flux of the charge current. \emph{Four:} each theory is
recoverable from any single one of its nonzero functions where one
exists globally (Theorem~\ref{thm:solutions};
Remark~\ref{rem:nontrivial} records when this is guaranteed), and
unconditionally from its local functions
(Proposition~\ref{prop:sheaf}, Corollary~\ref{cor:sheaf-classify}),
so the classification of sourceless
equations is the classification, one theory per equation, of the
pseudo-analytic function theories on the disk under geometric
isomorphism. Two numbers, cast by the moduli; the
moduli are fields.

\subsection*{Use of Generative AI Tools}
\medskip

The author discloses the use of Anthropic's Claude (Claude Fable 5,
accessed through the Claude.ai mobile interface in July 2026) in the
preparation of this manuscript. The tool was used as follows:

\begin{enumerate}
\item[(i)] \emph{Exploratory dialogue.} The organization of the paper
--- the minimal reduction and its residual groupoid, the splitting of
the orbit relation through the holomorphic square root of the
conformal weight, the pair $(\vartheta, \kappa)$ and the completeness
theorems, and the two-current structure of the phase at the vortices
--- emerged in iterative research sessions, building on the companion
papers \cite{mass, framed, charge}.

\item[(ii)] \emph{Symbolic and numeric verification.} The moment
integrals of Section~\ref{sec:incompleteness}, the equivariance
identities of Section~\ref{sec:pair}, and the vortex-current pairings
of Section~\ref{sec:vortices} were verified in exact rational and
floating-point computer algebra, with scripts produced with the
tool's assistance and re-run, and further verified and analyzed by
the author. In particular, the location of the charge atom in the
exterior rather than the co-exterior derivative of the phase form was
settled by these computations, correcting an earlier working
formulation; the displayed proofs, in particular
Theorem~\ref{thm:atoms}, are independent of these verifications.

\item[(iii)] \emph{Drafting and revision of prose.} The manuscript
was drafted in iterative dialogue; all claims and their precise
wording were reviewed by the author.
\end{enumerate}

The author takes full responsibility for the correctness, accuracy,
originality, and integrity of all content.

\subsection*{Disclosure of interest}

The author reports there are no competing interests to declare.

\end{document}